\documentclass[12pt, a4paper]{amsart}

\usepackage[T1]{fontenc}

\usepackage{amsmath, amssymb, amsthm}
\usepackage{mathtools}
\usepackage{latexsym}

\usepackage{amsbsy}

\usepackage[left=30mm, right=30mm, top=20mm, bottom=20mm]{geometry}

\usepackage{enumitem}
\usepackage{comment}

\usepackage{xcolor}

\usepackage{tikz}
\usetikzlibrary{circuits, intersections}
\usetikzlibrary{plotmarks}
\usetikzlibrary{angles, quotes}

\usepackage{graphicx}
\usepackage{pdfpages}

\usepackage[super]{nth}

\usepackage{hyperref}
\usepackage[colorinlistoftodos]{todonotes}

\newtheorem{thmintro}{Theorem}

\newtheorem{corintro}[thmintro]{Corollary}

\theoremstyle{definition}

\theoremstyle{plain}
\newtheorem{theorem}{Theorem}[section]
\newtheorem{proposition}[theorem]{Proposition}
\newtheorem{lemma}[theorem]{Lemma}

\theoremstyle{definition}

\newtheorem{definition}[theorem]{Definition}
\newtheorem{convention}[theorem]{Convention}
\newtheorem{example}[theorem]{Example}

\theoremstyle{remark}
\newtheorem{remark}[theorem]{Remark}

\renewcommand{\phi}{\varphi}
\renewcommand{\epsilon}{\varepsilon}

\newcommand{\FF}{\mathbb{F}}

\newcommand{\Cp}{\mathcal{C}_{\epsilon}}
\newcommand{\Cm}{\mathcal{C}_{-\epsilon}}
\newcommand{\C}{\mathcal{C}}
\renewcommand{\P}{\mathcal{P}}

\DeclareMathOperator{\Aut}{Aut}

\DeclareMathOperator{\proj}{proj}

\DeclareMathOperator{\Fix}{Fix}
\DeclareMathOperator{\Stab}{Stab}

\DeclareMathOperator{\equal}{=}

\numberwithin{equation}{section} 

\begin{document}
	
\title{Isomorphisms of Groups of Kac-Moody Type Over $\FF_2$}

\author{Sebastian Bischof}

\thanks{email: sebastian.bischof@math.uni-paderborn.de}

\thanks{UCLouvain, IRMP, Chemin du Cyclotron 2, 1348 Louvain-la-Neuve, Belgium}



\begin{abstract}
	In \cite{CM06} Caprace and M\"uhlherr solved the isomorphism problem for Kac-Moody groups of non-spherical type over finite fields of cardinality at least $4$. In this paper we solve the isomorphism problem for RGD-systems (e.g.\ Kac-Moody groups) over $\FF_2$ whose type is $2$-complete and $\tilde{A}_2$-free.
\end{abstract}

\maketitle

\section{Introduction}

Chevalley groups of type $\Phi$ (over fields) are abstract groups which come equipped with a family of subgroups $(U_{\alpha})_{\alpha \in \Phi}$ indexed by the root system $\Phi$ and satisfying the axioms of an \emph{RGD-system}. It is known that each automorphism of a Chevalley group (of irreducible type and over a perfect field) can be written as a product of an inner, a diagonal, a graph and a field automorphism (cf.\ \cite[Theorem $30$]{St16}). As a consequence of this result, any automorphism of a Chevalley group induces an automorphism of its underlying RGD-system.

Kac-Moody groups are infinite-dimensional generalizations of Chevalley groups. While in Chevalley groups opposite Borel subgroups are conjugate, this is no longer true in Kac-Moody groups of non-spherical type and a new class of automorphisms arises, the so-called \emph{sign automorphisms}. It was conjectured in \cite{KP87} that any automorphism of a Kac-Moody group over an algebraically closed field of characteristic $0$ can be written as a product of an inner, a diagonal, a graph, a field and a sign automorphism. This result was shown to be true in \cite{CM05b} by Caprace and M\"uhlherr for algebraically closed fields of any characteristic (and generalized by Caprace in \cite{Caprace_Diss}). They deduced this result from the solution of the isomorphism problem for Kac-Moody groups or, more generally, for RGD-systems.

\subsection*{Isomorphism problem for RGD-systems}

Given a Coxeter system $(W, S)$, an \emph{RGD-system of type $(W, S)$} is a pair $\mathcal{D} = \left(G, \left( U_{\alpha} \right)_{\alpha \in \Phi} \right)$ consisting of a group $G$ and a family of subgroups indexed by the set of roots $\Phi$ of $(W, S)$ satisfying some axioms (we refer to Section \ref{Section: Isomorphisms of RGD-systems} for the precise definition). The group $G$ is also denoted by $G^{\mathcal{D}}$ and the \emph{torus} of $\mathcal{D}$ is defined to be the subgroup $H := \bigcap_{\alpha \in \Phi} N_G(U_{\alpha})$.

The isomorphism problem for RGD-systems is the question whether, given two RGD-systems $\mathcal{D}$ of type $(W, S)$ and $\mathcal{D}'$ of type $(W', S')$, a group isomorphism $\phi: G^{\mathcal{D}} \to G^{\mathcal{D}'}$ induces an isomorphism from $\mathcal{D}$ to $\mathcal{D}'$. Caprace and M\"uhlherr have shown in \cite[Theorem $2.2$]{CM05b} that this is the case as soon as there exists $x\in G^{\mathcal{D}'}$ such that
\[ \{ \phi( U_{\alpha} ) \mid \alpha \in \Phi(W, S) \} = \{ x U_{\alpha}' x^{-1} \mid \alpha \in \Phi(W', S') \}. \]
This reduction is also used in \cite{CM06} to solve the isomorphism problem for Kac-Moody groups over finite fields of cardinality at least $4$. The proof makes a crucial use of the action of the torus on the associated twin building, and more precisely of the fact that the torus fixes no chamber outside the fundamental twin apartment. If we consider an RGD-system over $\FF_2$ (i.e.\ any root group has cardinality $2$), then the torus is trivial and fixes the whole twin building. In particular, we cannot follow the strategy of \cite{CM06}. We note that without any further condition the isomorphism problem has a negative answer for example witnessed by the two exceptional isomorphisms $\mathrm{PSL}_2(7) \cong \mathrm{PSL}_3(2)$ and $^2A_3(2) \cong C_2(3)$. Caprace and M\"uhlherr excluded these cases by assuming that the root groups are large enough. We will exclude these isomorphisms by assuming that all root groups have cardinality $2$.

In \cite{BiDiss} we have constructed uncountably many new examples of RGD-systems over $\FF_2$ of type $(4, 4, 4)$ and it would be most desirable to know that all these groups are pairwise non-isomorphic. This was originally our main motivation to solve the isomorphism problem for such RGD-systems. It turned out that our arguments work not only for the type $(4, 4, 4)$, but for Coxeter system $(W, S)$ which are \emph{$2$-complete} (i.e.\ $2 < o(st) < \infty$ for all $s\neq t \in S$) and $\tilde{A}_2$-free (for any $J\subseteq S$ the Coxeter system $(\langle J \rangle, J)$ is not of type $\tilde{A}_2$). Following \cite{Ca06}, we call an RGD-system $\left(G, \left( U_{\alpha} \right)_{\alpha \in \Phi} \right)$ \emph{centered} if $G = \langle U_{\alpha} \mid \alpha \in \Phi \rangle$.

\begin{thmintro}\label{Main theorem: isomorphism theorem}
	Suppose that $(W, S)$ and $(W', S')$ are $2$-complete and $\tilde{A}_2$-free Coxeter systems of finite rank at least $3$ and let $\mathcal{D}$ and $\mathcal{D}'$ be two centered RGD-systems over $\FF_2$ of type $(W, S)$ and $(W', S')$. Then any isomorphism $G^{\mathcal{D}} \to G^{\mathcal{D}'}$ induces an isomorphism of $\mathcal{D}$ to $\mathcal{D}'$.
\end{thmintro}

\begin{corintro}\label{Main Corollary: automorphism group}
	Suppose that $(W, S)$ is a $2$-complete and $\tilde{A}_2$-free Coxeter system of finite rank at least $3$ and let $\mathcal{D}$ be a centered RGD-system of type $(W, S)$ over $\FF_2$. Then any automorphism of $G$ is a product of an inner, a graph and a sign automorphism.
\end{corintro}

The proof of Theorem \ref{Main theorem: isomorphism theorem} is based on the following fact: Any reflection triangle (see Section \ref{Section: Coxeter buildings} for the precise definition) of a $2$-complete and $\tilde{A}_2$-free Coxeter system of rank $3$ is a chamber (cf.\ classification in \cite[Fig.\ $8$ in $\S 5.1$]{Fe98}). We first show that an analogous result holds in higher rank (cf.\ Theorem \ref{Theorem: combinatorial triangle is chamber}). Then we introduce the notion of \emph{triangles} in general buildings which can be seen as generalization of reflection triangles from apartments to buildings and we obtained the following generalization of Theorem \ref{Theorem: combinatorial triangle is chamber} (cf.\ Theorem \ref{Theorem: 2-complete and A_2 tilde free triangle contains unique chamber}):

\begin{thmintro}\label{Main Theorem: triangle}
	Let $(W, S)$ be a $2$-complete and $\tilde{A}_2$-free Coxeter system of finite rank at least $3$, let $\Delta = (\C, \delta)$ be a building of type $(W, S)$ and let $\{ R_1, R_2, R_3 \}$ be a triangle. Then $R_1 \cap R_2 \cap R_3$ is non-empty and contains a unique chamber.
\end{thmintro}

\subsection*{Overview}

Section \ref{Section: Premilinaries} is devoted to fixing notation, to recalling some known facts and to proving some useful and elementary results about buildings. In Section \ref{Section: Coxeter buildings} we study reflection and combinatorial triangles in Coxeter buildings. The goal of this section is to prove Theorem \ref{Theorem: combinatorial triangle is chamber}. In Section \ref{Section: Triangles in buildings} we introduce the notion of \emph{triangles} and prove some result about them. We end the section by proving Theorem \ref{Main Theorem: triangle}. In Section \ref{Section: Isomorphisms of RGD-systems} we recall the definitions of twin buildings and RGD-systems and we recall or prove some results about them. In Section \ref{Section: Main result} we put everything together and solve the isomorphism problem for centered RGD-systems over $\FF_2$ of $2$-complete and $\tilde{A}_2$-free type.

\subsection*{Acknowledgement}

I am grateful to Bernhard M\"uhlherr for drawing my attention to this problem. I thank him and Fran\c{c}ois Thilmany for many helpful discussions on the topic. I also thank Timothée Marquis for valuable remarks on an earlier draft. This work was supported by a fellowship of the German Academic Exchange Service (DAAD) via the grant 57664192.

\section{Preliminaries}\label{Section: Premilinaries}

\subsection*{Coxeter systems}

Let $(W, S)$ be a Coxeter system and let $\ell$ denote the corresponding length function. For $s, t \in S$ we denote the order of $st$ in $W$ by $m_{st}$. The \emph{Coxeter diagram} corresponding to $(W, S)$ is the labeled graph $(S, E(S))$, where $E(S) = \{ \{s, t \} \mid m_{st}>2 \}$ and where each edge $\{s,t\}$ is labeled by $m_{st}$ for all $s, t \in S$. The \emph{rank} of the Coxeter system is the cardinality of the set $S$.

It is well-known that for each $J \subseteq S$ the pair $(\langle J \rangle, J)$ is a Coxeter system (cf.\ \cite[Ch.\ IV, §$1$ Theorem $2$]{Bo68}). A subset $J \subseteq S$ is called \emph{spherical} if $\langle J \rangle$ is finite. Given a spherical subset $J$ of $S$, there exists a unique element of maximal length in $\langle J \rangle$, which we denote by $r_J$ (cf.\ \cite[Corollary $2.19$]{AB08}). The Coxeter system $(W, S)$ is called \emph{spherical} if $S$ is spherical; it is called \emph{$2$-spherical} if $\langle J \rangle$ is finite for each $J \subseteq S$ with $\vert J \vert \leq 2$ (i.e.\ $m_{st} < \infty$ for all $s, t \in S$).

\begin{definition}
	Let $(W, S)$ be a Coxeter system.
	\begin{enumerate}[label=(\alph*)]
		\item $(W, S)$ is called \emph{$2$-complete}, if it is $2$-spherical and if the underlying Coxeter diagram is the complete graph.
		
		\item $(W, S)$ is called \emph{$\tilde{A}_2$-free}, if for each $J \subseteq S$ the Coxeter system $(\langle J \rangle, J)$ is not of type $\tilde{A}_2$.
	\end{enumerate}
\end{definition}

\begin{lemma}\label{Lemma: BiCoxGrowth2.8}
	Let $(W, S)$ be a $2$-complete Coxeter system of finite rank. Suppose $w\in W$ and $s\neq t \in S$ with $\ell(ws) = \ell(w) +1 = \ell(wt)$ and suppose $w' \in \langle s, t \rangle$ with $\ell(w') \geq 2$. Then we have $\ell(ww'r) = \ell(w) + \ell(w') +1$ for each $r\in S \backslash \{s, t\}$.
\end{lemma}
\begin{proof}
	This is \cite[Corollary $2.8$]{BiCoxGrowth}.
\end{proof}

\subsection*{Buildings}

Let $(W, S)$ be a Coxeter system. A \emph{building of type $(W, S)$} is a pair $\Delta = (\C, \delta)$ where $\C$ is a non-empty set and where $\delta: \C \times \C \to W$ is a \emph{distance function} satisfying the following axioms, where $x, y\in \C$ and $w = \delta(x, y)$:
\begin{enumerate}[label=(Bu\arabic*)]
	\item $w = 1_W$ if and only if $x=y$;
	
	\item if $z\in \C$ satisfies $s := \delta(y, z) \in S$, then $\delta(x, z) \in \{w, ws\}$, and if, furthermore, $\ell(ws) = \ell(w) +1$, then $\delta(x, z) = ws$;
	
	\item if $s\in S$, there exists $z\in \C$ such that $\delta(y, z) = s$ and $\delta(x, z) = ws$.
\end{enumerate}
The \emph{rank} of $\Delta$ is the rank of the underlying Coxeter system. The elements of $\C$ are called \emph{chambers}. Given $s\in S$ and $x, y \in \C$, then $x$ is called \emph{$s$-adjacent} to $y$, if $\delta(x, y) = s$. The chambers $x, y$ are called \emph{adjacent}, if they are $s$-adjacent for some $s\in S$. A \emph{gallery} from $x$ to $y$ is a sequence $(x = x_0, \ldots, x_k = y)$ such that $x_{l-1}$ and $x_l$ are adjacent for all $1 \leq l \leq k$; the number $k$ is called the \emph{length} of the gallery. A gallery from $x$ to $y$ of length $k$ is called \emph{minimal} if there is no gallery from $x$ to $y$ of length $<k$. For two chambers $x$ and $y$ we define $\ell(x, y) := \ell(\delta(x, y))$.

Given a subset $J \subseteq S$ and $x\in \C$, the \emph{$J$-residue} of $x$ is the set $R_J(x) := \{y \in \C \mid \delta(x, y) \in \langle J \rangle \}$. Each $J$-residue is a building of type $(\langle J \rangle, J)$ with the distance function induced by $\delta$ (cf.\ \cite[Corollary $5.30$]{AB08}). A \emph{residue} is a subset $R$ of $\C$ such that there exist $J \subseteq S$ and $x\in \C$ with $R = R_J(x)$. Since the subset $J$ is uniquely determined by $R$, the set $J$ is called the \emph{type} of $R$ and the \emph{rank} of $R$ is defined to be the cardinality of $J$. Given $x\in \C$ and a $J$-residue $R$, then there exists a unique chamber $z\in R$ such that $\ell(x, y) = \ell(x, z) + \ell(z, y)$ holds for every $y\in R$ (cf.\ \cite[Proposition $5.34$]{AB08}). The chamber $z$ is called the \textit{projection of $x$ onto $R$} and is denoted by $\proj_R x$. Moreover, if $z = \proj_R x$ we have $\delta(x, y) = \delta(x, z) \delta(z, y)$ for each $y\in R$. A residue is called \emph{spherical} if its type is a spherical subset of $S$. Let $R$ be a spherical $J$-residue. Then $x, y \in R$ are called \emph{opposite in $R$} if $\delta(x, y) = r_J$. Two residues $P, Q \subseteq R$ are called \emph{opposite in $R$} if for each $p\in P$ there exists $q\in Q$ such that $p, q$ are opposite in $R$. A \emph{panel} is a residue of rank $1$ and we define $\P_s(c) := R_{\{s\}}(c)$ for all $(c, s) \in \C \times S$. The building $\Delta$ is called \emph{thick}, if each panel of $\Delta$ contains at least three chambers; it is called \emph{spherical} if its type is spherical.

\begin{lemma}\label{Lemma: Equal projection chamber on residues}
	Let $\Delta = (\C, \delta)$ be a building of type $(W, S)$. Let $R$ and $Q$ be two residues of $\Delta$ with $Q \subseteq R$, and let $c\in \C$ with $\proj_R c \in Q$. Then $\proj_R c = \proj_Q c$.
\end{lemma}
\begin{proof}
	This follows from the following easy computation (the first equation follows from the fact $Q \subseteq R$; the second equation follows from the fact $\proj_R c \in Q$):
	\allowdisplaybreaks
	\begin{align*}
		\ell(c, \proj_Q c) &= \ell(c, \proj_R c) + \ell(\proj_R c, \proj_Q c) \\
		&= \ell(c, \proj_Q c) + \ell(\proj_Q c, \proj_R c) + \ell(\proj_R c, \proj_Q c) \\
		&= \ell(c, \proj_Q c) + 2 \cdot \ell(\proj_Q c, \proj_R c) \qedhere
	\end{align*}
\end{proof}

Let $\Delta = (\C, \delta)$ and $\Delta' = (\C', \delta')$ be two buildings of type $(W, S)$. An \emph{isometry} between $\mathcal{X} \subseteq \C$ and $\mathcal{X}' \subseteq \C'$ is a bijection $\phi: \mathcal{X} \to \mathcal{X}'$ such that $\delta'(\phi(x), \phi(y)) = \delta(x, y)$ holds for all $x, y \in \mathcal{X}$. In this case $\mathcal{X}$ and $\mathcal{X}'$ are called \emph{isometric}. An \emph{(type-preserving) automorphism} of a building $\Delta = (\C, \delta)$ is an isometry $\phi:\C \to \C$. We remark that some authors distinguish between automorphisms and type-preserving automorphisms. An automorphism in our sense is type-preserving. We denote the set of all automorphisms of the building $\Delta$ by $\Aut(\Delta)$.

\begin{example}
	We define $\delta: W \times W \to W, (x, y) \mapsto x^{-1}y$. Then $\Sigma(W, S) := (W, \delta)$ is a building of type $(W, S)$. The group $W$ acts faithfully on $\Sigma(W, S)$ by multiplication from the left, i.e.\ $W \leq \Aut(\Sigma(W, S))$.
\end{example}

\subsection*{Apartments}

Let $(W, S)$ be a Coxeter system and let $\Delta = (\C, \delta)$ be a building of type $(W, S)$. A subset $\Sigma \subseteq \C$ is called \emph{convex} if $\proj_P c \in \Sigma$ for all $c\in \Sigma$ and each panel $P \subseteq \C$ which meets $\Sigma$. We note that if $\Sigma \subseteq \C$ is a convex subset and if $R$ is a residue which meets $\Sigma$, then $\proj_R c \in \Sigma$ for all $c\in \Sigma$ (cf.\ \cite[Lemma $5.45$]{AB08}). A subset $\Sigma \subseteq \C$ is called \emph{thin} if $P \cap \Sigma$ contains exactly two chambers for each panel $P \subseteq \C$ which meets $\Sigma$. An \emph{apartment} is a non-empty subset $\Sigma \subseteq \C$, which is convex and thin. Moreover, a subset $\Sigma \subseteq \C$ is an apartment if and only if it is isometric to $\Sigma(W, S)$ (cf.\ \cite[Corollary $5.67$]{AB08}). Furthermore, any subset of $\C$ that is isometric to a subset of $W$ is contained in an apartment (cf.\ \cite[Theorem $5.73$]{AB08}). As a consequence we note that for any two chambers there exists an apartment containing both (cf.\ \cite[Corollary $5.74$]{AB08}).

\subsection*{Roots}

Let $(W, S)$ be a Coxeter system. A \emph{reflection} is an element of $W$ that is conjugate to an element of $S$. For $s\in S$ we let $\alpha_s := \{ w\in W \mid \ell(sw) > \ell(w) \}$ be the \emph{simple root} corresponding to $s$. A \emph{root} is a subset $\alpha \subseteq W$ such that $\alpha = v\alpha_s$ for some $v\in W$ and $s\in S$. We denote the set of all roots by $\Phi(W, S)$. The set $\Phi(W, S)_+ := \{ \alpha \in \Phi(W, S) \mid 1_W \in \alpha \}$ is the set of all \emph{positive roots} and $\Phi(W, S)_- := \{ \alpha \in \Phi(W, S) \mid 1_W \notin \alpha \}$ is the set of all \emph{negative roots}. For each root $\alpha \in \Phi(W, S)$, the complement $-\alpha := W \backslash \alpha$ is again a root; it is called the root \emph{opposite} to $\alpha$. We denote the unique reflection which interchanges these two roots by $r_{\alpha} \in W \leq \Aut(\Sigma(W, S))$. We note that roots are convex (cf.\ \cite[Lemma $3.44$]{AB08}). A pair $\{ \alpha, \beta \}$ of roots is called \emph{prenilpotent} if both $\alpha \cap \beta$ and $(-\alpha) \cap (-\beta)$ are non-empty sets. For such a pair we will write $\left[ \alpha, \beta \right] := \{ \gamma \in \Phi(W, S) \mid \alpha \cap \beta \subseteq \gamma \text{ and } (-\alpha) \cap (-\beta) \subseteq -\gamma \}$ and $(\alpha, \beta) := \left[ \alpha, \beta \right] \backslash \{ \alpha, \beta \}$.

Let $\Delta = (\C, \delta)$ be a building of type $(W, S)$. A subset $\alpha \subseteq \C$ is called \emph{root}, if it is isometric to $\alpha_s$ for some $s\in S$. We denote the set of roots contained in some apartment $\Sigma$ by $\Phi^{\Sigma}$. We note that we can identify $\Phi^{\Sigma}$ with $\Phi(W, S)$.

\begin{convention}
	For the rest of this paper we let $(W, S)$ be a Coxeter system of finite rank and we define $\Phi := \Phi(W, S)$ (resp.\ $\Phi_+, \Phi_-$). 
\end{convention}

\begin{lemma}\label{Lemma: no apartment contains chambers}
	Let $\Delta = (\C, \delta)$ be a building of type $(W, S)$. Let $(d_0, \ldots, d_k)$ be a minimal gallery and let $c\in \C$. Let $\Sigma$ be an apartment containing $c, d_0, \ldots, d_{k-1}$ and suppose that no apartment contains $c, d_0, \ldots, d_k$. Suppose $e\in \Sigma$ with $\delta(d_{k-1}, e) = \delta(d_{k-1}, d_k)$ and let $\alpha \in \Phi^{\Sigma}$ be the unique root containing $d_{k-1}$ but not $e$. Then $c\notin \alpha$.
\end{lemma}
\begin{proof}
	By contrary we assume $c\in \alpha$. Define $s := \delta(d_{k-1}, d_k)$ and $P := \P_s(d_k)$. As $(d_0, \ldots, d_k)$ is a minimal gallery, we have $\proj_P d_0 = d_{k-1}$. Note that $P \cap \alpha = \{d_{k-1}\}$. As $c, d_{k-1} \in \alpha$ and roots are convex, we deduce $\proj_P c \in \alpha \cap P = \{ d_{k-1} \}$. As apartments are isometric to $\Sigma(W, S)$ we have the following:
	\[ \delta(c, d_k) = \delta(c, d_{k-1}) s = \delta(c, d_0) \delta(d_0, d_{k-1})s = \delta(c, d_0) \delta(d_0, d_k). \]
	But this implies that $\{ c, d_0, d_k \} \subseteq \C$ is isometric to a subset of $W$. We conclude that there exists an apartment containing $c, d_0, d_k$ (cf.\ also \cite[Exercise $5.77(a)$]{AB08}). This is a contradiction and we infer $c \notin \alpha$.
\end{proof}

\subsection*{Parallel residues in buildings}

In this subsection we let $\Delta = (\C, \delta)$ be a building of type $(W, S)$. For two residues $R, T$ we define the mapping $\proj^R_T: R \to T, x \mapsto \proj_T x$ and we define $\proj_T R := \{ \proj_T r \mid r\in R \}$. We note that $\proj_T R$ is a residue contained in $T$ (cf.\ \cite[Lemma $5.36(2)$]{AB08}). The residues $R, T$ are called \textit{parallel} if $\proj_R T = R$ and $\proj_T R = T$. We note that for parallel residues $R$ and $T$ the element $\delta(c, \proj_T c)$ is independent of the choice of $c\in R$ (cf.\ \cite[Proposition $21.10$]{MPW15}).

\begin{lemma}\label{Lemma: results about parallel residues}
	Let $R$ and $T$ be two residues. Then the following hold:
	\begin{enumerate}[label=(\alph*)]
		\item The residues $\proj_R T$ and $\proj_T R$ are parallel.
		
		\item $R$ and $T$ are parallel if and only if $\proj_T^R$ and $\proj_R^T$ are bijections inverse to each other.
		
		\item Let $\Sigma$ be an apartment containing chambers of $R$ and $T$. Then $R$ and $T$ are parallel if and only if $R\cap \Sigma$ and $T\cap \Sigma$ are parallel residues of $\Sigma$.
	\end{enumerate}
\end{lemma}
\begin{proof}
	Part $(a)$ is \cite[Proposition $21.8(i)$]{MPW15}. One implication of part $(b)$ is easy; the other is \cite[Proposition $21.10(i)$]{MPW15}. Part $(c)$ is \cite[Proposition $21.17$]{MPW15}.
\end{proof}

\begin{lemma}\label{Lemma: triangle projection}
	Let $R$ and $T$ be two residues and let $c\in \proj_R T$. Then we have $\proj_R ( \proj_T c ) = c$.
\end{lemma}
\begin{proof}
	Note that by Lemma \ref{Lemma: results about parallel residues} the residues $R' := \proj_R T$ and $T' :=\proj_T R$ are parallel and we have $\proj_{R'} ( \proj_{T'} c ) = c$. It follows now from Lemma \ref{Lemma: Equal projection chamber on residues} that $\proj_{T'} c = \proj_T c$ and $\proj_{R'} (\proj_T c) = \proj_R (\proj_T c)$.
\end{proof}

\begin{lemma}\label{Lemma: residues and apartments}
	Let $R$ and $T$ be two residues and let $\Sigma$ be an apartment with $\Sigma \cap R \neq \emptyset \neq \Sigma \cap T$. Then the following hold:
	\begin{enumerate}[label=(\alph*)]
		\item $\Sigma \cap \proj_R T = \proj_{\Sigma \cap R} ( \Sigma \cap T )$;
		
		\item $\Sigma \cap \proj_R T$ is a residue of $\Sigma$ having the same type as $\proj_R T$.
	\end{enumerate}
\end{lemma}
\begin{proof}
	We first note that for $c\in \Sigma$ we have $\proj_{\Sigma \cap R} c = \proj_R c$, as $\Sigma$ is convex. This implies $\proj_{\Sigma \cap R} ( \Sigma \cap T ) \subseteq \Sigma \cap \proj_R T$. For the other inclusion let $c\in \Sigma \cap \proj_R T$ be a chamber. By Lemma \ref{Lemma: triangle projection} we have $\proj_R ( \proj_T c ) = c$. As $\Sigma$ is convex and $c\in \Sigma$, we have $c' := \proj_T c \in \Sigma \cap T$. But then $c = \proj_R c' = \proj_{\Sigma \cap R} c' \in \proj_{\Sigma \cap R} ( \Sigma \cap T )$. This shows $(a)$. Part $(b)$ follows from $(a)$, as $\Sigma \cap \proj_R T \neq \emptyset$.
\end{proof}

\subsection*{Compatible paths and parallel panels}

Compatible paths were introduced in \cite{DMVM11}. Let $\Delta = (\C, \delta)$ be a building of type $(W, S)$ and let $\Gamma$ be the graph whose vertices are the panels of $\Delta$ and in which two panels form an edge if and only if there exists a rank $2$ residue in which the two panels are opposite. For two adjacent panels $P, Q$, there exists a unique rank $2$ residue containing $P$ and $Q$, which will be denoted by $R(P, Q)$. A path $\gamma = (P_0, \ldots, P_k)$ in $\Gamma$ is called \textit{compatible} if $\proj_{R(P_{i-1}, P_i)} P_0 = P_{i-1}$ holds for all $1 \leq i \leq k$. Let $(P_0, \ldots, P_k)$ be a compatible path. By definition, $(P_0, \ldots, P_i)$ is a compatible path for all $0 \leq i \leq k$. Moreover, $(P_k, \ldots, P_0)$ is a compatible path as well (cf.\ \cite[Proposition $(4.2)(d)$]{BM23}).

\begin{lemma}\label{Lemma: Projection on compatible paths}
	Let $m>1$ and let $(P_0, \ldots, P_m)$ be a compatible path. Then we have $\proj_{R(P_{m-1}, P_m)} R(P_i, P_{i+1}) = \proj_{R(P_{m-1}, P_m)} R(P_0, P_1)$ for all $0 \leq i \leq m-2$.
\end{lemma}
\begin{proof}
	The claim is trivial for $m=2$. Thus we can assume $m>2$. As $(P_i, \ldots, P_m)$ is again a compatible path, it suffices to show $\proj_{R(P_{m-1}, P_m)} R(P_0, P_1) = P_{m-1}$. We define $R_j := R(P_{j-1}, P_j)$ for all $1 \leq j \leq m$. Clearly, we have $P_{m-1} = \proj_{R_m} P_0 \subseteq \proj_{R_m} R_1$. Assume that $\proj_{R_m} R_1 \supsetneq P_{m-1}$. Then, as $\proj_{R_m} R_1$ is a residue, we would have $\proj_{R_m} R_1 = R_m$. By Lemma \ref{Lemma: results about parallel residues}, the residues $\proj_{R_m} R_1 = R_m$ and $R:= \proj_{R_1} R_m$ are parallel and $\proj_R^{R_m}$ is a bijection. But $(P_m, \ldots, P_0)$ and $(P_{m-1}, \ldots, P_0)$ are compatible paths and hence, as $m>1$, $\proj_{R(P_0, P_1)} P_m = P_1 = \proj_{R(P_0, P_1)} P_{m-1}$. This is a contradiction.
\end{proof}

\begin{lemma}\label{Lemma: results about parallel panels}
	Let $P$ and $Q$ be two panels. Then the following hold:
	\begin{enumerate}[label=(\alph*)]
		\item The following are equivalent:
		\begin{enumerate}[label=(\roman*)]
			\item $P$ and $Q$ are parallel.
			
			\item $\vert \proj_Q P \vert \geq 2$;
			
			\item There exists a compatible path from $P$ to $Q$.
		\end{enumerate}
		
		\item If $P$ and $Q$ are parallel and if $R$ is a residue containing $Q$, then $\proj_R P$ is a panel parallel to both $P$ and $Q$.
	\end{enumerate}
\end{lemma}
\begin{proof}
	Part $(a)$ follows from \cite[Lemma $13$ and Lemma $19$]{DMVM11}; part $(b)$ is \cite[Lemma $17$]{DMVM11}.
\end{proof}

\begin{lemma}\label{Lemma: no combinatorial triangle}
	Suppose that $(W, S)$ is $2$-complete. Let $m>0$, let $(P_0, \ldots, P_m)$ be a compatible path and let $c\in \C$. Suppose that the following two conditions hold:
	\begin{enumerate}[label=(\roman*)]
		\item $\ell( c, \proj_{P_0} c ) < \ell( c, \proj_{P_1} c )$;
		
		\item $\ell( \proj_{R(P_0, P_1)} c, \proj_{P_1} c ) \geq 2$.
	\end{enumerate}
	Then $\ell( c, \proj_{P_0} c ) < \ell( c, \proj_{P_m} c )$ and $\ell( \proj_{R(P_{m-1}, P_m)} c, \proj_{P_m} c ) \geq 2$ hold.
\end{lemma}
\begin{proof}
	We will show the claim by induction on $m$. For $m=1$ this is exactly the assumption. Thus we can assume $m>1$. Using induction the following hold:
	\begin{itemize}
		\item $\ell( c, \proj_{P_0} c ) < \ell( c, \proj_{P_{m-1}} c )$;
		\item $\ell( \proj_{R(P_{m-2}, P_{m-1})} c, \proj_{P_{m-1}} c ) \geq 2$.
	\end{itemize}
	Now Lemma \ref{Lemma: BiCoxGrowth2.8} yields $\proj_{R(P_{m-1}, P_m)} c = \proj_{P_{m-1}} c$ and we infer
	\allowdisplaybreaks
	\begin{align*}
		\ell( c, \proj_{P_m} c ) &= \ell( c, \proj_{R(P_{m-1}, P_m)} c ) + \ell( \proj_{R(P_{m-1}, P_m)} c, \proj_{P_m} c ) \\
		&\geq \ell( c, \proj_{P_{m-1}} c ) \\
		&> \ell( c, \proj_{P_0} c )
	\end{align*}
	Note that $P_{m-1}$ and $P_m$ are opposite in $R(P_{m-1}, P_m)$. As $(W, S)$ is $2$-complete, we conclude $\ell( \proj_{R(P_{m-1}, P_m)} c, \proj_{P_m} c ) \geq 2$. This finishes the proof.
\end{proof}

\begin{lemma}\label{Lemma: concatenation of compatible paths}
	Suppose that $(W, S)$ is $2$-complete. Let $(P_0, \ldots, P_m)$ and $(Q_0, \ldots, Q_n)$ be two compatible paths such that $P_m$ and $Q_0$ are opposite in a rank $2$ residue. If $\proj_{R(P_m, Q_0)} P_0 = P_m$, then $(P_0, \ldots, P_m, Q_0, \ldots, Q_n)$ is a compatible path.
\end{lemma}
\begin{proof}
	We show the claim by induction on $n$. For $n=0$ the claim follows by definition. Thus we can assume $n>0$. Using induction, $(P_0, \ldots, P_m, Q_0, \ldots, Q_{n-1})$ is a compatible path. We define $P_{m+i+1} := Q_i$ for $0 \leq i \leq n$. We have to show that $\proj_{R(P_{m+n}, P_{m+n+1})} P_0 = P_{m+n}$. As $(P_0, \ldots, P_{m+n})$ is a compatible path, we have $\proj_{R(P_{m+n-1}, P_{m+n})} P_0 = P_{m+n-1}$. As $P_{m+n-1}$ and $P_{m+n}$ are opposite in $R(P_{m+n-1}, P_{m+n})$, we deduce from Lemma \ref{Lemma: BiCoxGrowth2.8} that $\proj_{R(P_{m+n}, P_{m+n+1})} p = \proj_{P_{m+n}} p$ holds for all $p \in P_0$. As $P_0$ and $P_{m+n}$ are parallel by Lemma \ref{Lemma: results about parallel panels}$(a)$, we infer $\proj_{R(P_{m+n}, P_{m+n+1})} P_0 = \proj_{P_{m+n}} P_0 = P_{m+n}$. This finishes the proof.
\end{proof}

\begin{lemma}\label{Lemma: parallel panels - word become longer}
	Suppose that $(W, S)$ is $2$-complete. Let $m>0$, let $(P_0, \ldots, P_m)$ be a compatible path, let $c\in P_0$ and let $d \in R(P_{m-1}, P_m) \backslash P_{m-1}$. Let $\{s, t\}$ be the type of $R(P_{m-1}, P_m)$ for $s\neq t \in S$. Then $\ell(\delta(c, d)r) = \ell(c, d) +1$ holds for each $r\in S \backslash \{s, t\}$.
\end{lemma}
\begin{proof}
	By Lemma \ref{Lemma: BiCoxGrowth2.8} it suffices to show the claim for $d$ adjacent to $\proj_{R(P_{m-1}, P_m)} c$. We prove the claim by induction on $m$. For $m=1$ the claim follows from \cite[Lemma $2.15$]{AB08}. Suppose now $m>1$. Let $s\in S$ be the type of $P_{m-1}$ and let $\{ s, q \}$ be the type of $R(P_{m-2}, P_{m-1})$. We define $d' := \proj_{\P_q(e)} c$ where $e := \proj_{R(P_{m-1}, P_m)} c$. Then $d' \in R(P_{m-2}, P_{m-1}) \backslash P_{m-2}$. Using induction, we have
	\[ \ell(\delta(c, d')r') = \ell(c, d') +1 \]
	for each $r' \in S \backslash \{ s, q \}$. Now let $r \in S \backslash \{ s, t \}$. If $r = q$, then the claim follows from the previous equation with $r'=t$ and the fact that $m_{rt} >2$. If $r\neq q$, then the previous equation implies $\ell( \delta(c, e ) r) = \ell(c, e) +1$ which itself yields the claim.
\end{proof}

\begin{lemma}\label{Lemma: stabilized implies parallel}
	Let $P$ and $Q$ be two panels and let $H \leq \Stab(P) \cap \Stab(Q)$ be a subgroup which does not fix a chamber in $P$. Then the following hold:
	\begin{enumerate}[label=(\alph*)]
		\item The panels $P$ and $Q$ are parallel.
		
		\item For a residue $R$ containing $Q$, $\proj_R P$ is a panel and $H \leq \Stab(\proj_R P)$.
	\end{enumerate}
\end{lemma}
\begin{proof}
	Let $q \in Q$ and put $p := \proj_P q$. As $H \leq \Stab(P) \cap \Stab(Q)$ does not fix $p$, there exists $h \in H$ with $h.p \neq p$. Using the uniqueness of the projection chamber, we infer
	\[ \proj_P (h.q) = \proj_{h.P} (h.q) = h.p \neq p \]
	Thus $\vert \proj_P Q \vert \geq 2$ and Lemma \ref{Lemma: results about parallel panels}$(a)$ implies that $P$ and $Q$ are parallel.
	
	Now let $R$ be a residue containing $Q$. Then $\proj_RP$ is a panel parallel to both $P$ and $Q$ by Lemma \ref{Lemma: results about parallel panels}$(b)$. Let $x\in \proj_R P$ and $h\in H$. We have to show that $h.x \in \proj_R P$. There exists $p \in P$ with $x = \proj_R p$. Using the uniqueness of the projection chamber and the fact that $H \leq \Stab(P) \cap \Stab(Q) \leq \Stab(P) \cap \Stab(R)$, we conclude
	\[ h.x = h.(\proj_R p) = \proj_{h.R} (h.p) = \proj_R (h.p) \in \proj_R P. \qedhere \]
\end{proof}

\section{Coxeter buildings}\label{Section: Coxeter buildings}

In this section we consider the Coxeter building $\Sigma(W, S)$. For $\alpha \in \Phi$ we denote by $\partial \alpha$ (resp.\ $\partial^2 \alpha$) the set of all panels (resp.\ spherical residues of rank $2$) stabilized by $r_{\alpha}$. We note that for $\alpha, \beta \in \Phi$ with $\alpha \neq \pm \beta$ we have $o(r_{\alpha} r_{\beta}) < \infty$ if and only if $\partial^2 \alpha \cap \partial^2 \beta \neq \emptyset$ (cf. \cite[Lemma $2.8(a)$]{BiConstruction}).

The set $\partial \alpha$ is called the \emph{wall} associated with $\alpha$. Let $G = (c_0, \ldots, c_k)$ be a gallery. We say that $G$ \emph{crosses the wall $\partial \alpha$} if there exists $1 \leq i \leq k$ such that $\{ c_{i-1}, c_i \} \in \partial \alpha$. It is a basic fact that a minimal gallery crosses a wall at most once (cf.\ \cite[Lemma $3.69$]{AB08}). Moreover, a gallery which crosses each wall at most once is already minimal.

\begin{lemma}\label{Lemma: CM06Prop2.7}
	Let $R, T$ be two spherical residues of $\Sigma(W, S)$. Then the following are equivalent:
	\begin{enumerate}[label=(\roman*)]
		\item $R$ and $T$ are parallel;
		
		\item a reflection of $\Sigma(W, S)$ stabilizes $R$ if and only if it stabilizes $T$;
		
		\item there exist two sequences $R_0 = R, \ldots, R_n = T$ and $T_1, \ldots, T_n$ of residues of spherical type such that for each $1 \leq i \leq n$ the rank of $T_i$ is equal to $1+\mathrm{rank}(R)$, the residues $R_{i-1}, R_i$ are contained and opposite in $T_i$ and moreover, we have $\proj_{T_i} R = R_{i-1}$ and $\proj_{T_i} T = R_i$.
	\end{enumerate}
\end{lemma}
\begin{proof}
	This is \cite[Proposition $2.7$]{CM06}.
\end{proof}

\begin{lemma}\label{Lemma: 2-dimensional rank 2 residues not parallel}
	Suppose $\vert \langle J \rangle \vert = \infty$ for all $J \subseteq S$ containing at least three elements.
	\begin{enumerate}[label=(\alph*)]
		\item Let $R$ and $T$ be two different spherical residues of $\Sigma(W, S)$ of rank $2$. Then $R$ and $T$ are not parallel.
		
		\item Let $\alpha, \beta \in \Phi$ with $\alpha \neq \pm \beta$. Then $\vert \partial^2 \alpha \cap \partial^2 \beta \vert \leq 1$.
	\end{enumerate}
\end{lemma}
\begin{proof}
	Part $(a)$ follows from Lemma \ref{Lemma: CM06Prop2.7} and $(b)$ follows from \cite[Lemma $2.5$]{BiCoxGrowth}.
\end{proof}

\subsection*{Reflection and combinatorial triangles in $\mathbf{\Sigma(W, S)}$}

A \textit{reflection triangle} is a set $T$ of three reflections such that the order of $tt'$ is finite for all $t, t' \in T$ and such that $\bigcap_{t\in T} \partial^2 \beta_t = \emptyset$, where $\beta_t$ is one of the two roots associated with the reflection $t$. Note that $\partial^2 \beta_t = \partial^2 (-\beta_t)$. A set of three roots $T$ is called \textit{combinatorial triangle} if the following hold:

\begin{enumerate}[label=(CT\arabic*)]
	\item The set $\{ r_{\alpha} \mid \alpha \in T \}$ is a reflection triangle.
	
	\item For each $\alpha \in T$, there exists $\sigma \in \partial^2 \beta \cap \partial^2 \gamma$ with $\sigma \subseteq \alpha$ where $\{ \beta, \gamma \} = T \backslash \{ \alpha \}$.
\end{enumerate}

\begin{remark}\label{Remark: reflection triangle plus orientation yields triangle}
	Let $R$ be a reflection triangle. Then there exist three roots $\beta_1, \beta_2, \beta_3 \in \Phi$ such that $R = \{ r_{\beta_1}, r_{\beta_2}, r_{\beta_3} \}$. Let $\{i, j, k\} = \{1, 2, 3\}$. As $o(r_{\beta_i} r_{\beta_j}) < \infty$, there exists $\sigma_k \in \partial^2 \beta_i \cap \partial^2 \beta_j$. Since $R$ is a reflection triangle, we have $\sigma_k \notin \partial^2 \beta_k$. Now \cite[Lemma $2.2$]{BiCoxGrowth} yields $\sigma_k \subseteq \beta_k$ or $\sigma_k \subseteq -\beta_k$. Let $\epsilon_k \in \{+, -\}$ with $\sigma_k \subseteq \epsilon_k \beta_k$ and define $\alpha_k := \epsilon_k \beta_k$. Then $\{ \alpha_1, \alpha_2, \alpha_3 \}$ is a combinatorial triangle, which induces the reflection triangle $R$.
\end{remark}

\begin{lemma}\label{Lemma: combinatorial triangle yields triangle}
	Let $\{ \alpha_1, \alpha_2, \alpha_3 \}$ be a combinatorial triangle and let $R_k \in \partial^2 \alpha_i \cap \partial^2 \alpha_j$ for $\{i, j, k\} = \{1, 2, 3\}$. Then the following hold:
	\begin{enumerate}[label=(\alph*)]
		\item The residue $\proj_{R_i} R_j$ is a panel which is contained in $\partial \alpha_k$.
		
		\item The panels $\proj_{R_k} R_i$ and $\proj_{R_k} R_j$ are not parallel.
	\end{enumerate}
\end{lemma}
\begin{proof}
	We know that $\proj_{R_i} R_j$ is a residue. As $R_i, R_j \in \partial^2 \alpha_k$, there are panels $P_i \subseteq R_i$ and $P_j \subseteq R_j$ with $P_i, P_j \in \partial \alpha_k$. In particular, $P_i$ and $P_j$ are parallel by Lemma \ref{Lemma: CM06Prop2.7} and $\proj_{R_i} R_j$ contains the panel $\proj_{R_i} P_j$ (cf.\ Lemma \ref{Lemma: results about parallel panels}$(c)$). Suppose $\proj_{R_i} R_j$ is not a panel. Then $\proj_{R_i} R_j = R_i$. As $R_i = \proj_{R_i} R_j$ and $\proj_{R_j} R_i$ are parallel by Lemma \ref{Lemma: results about parallel residues}$(a)$, it follows $\proj_{R_j} R_i = R_j$ by Lemma \ref{Lemma: CM06Prop2.7} (as they must have the same rank). But then $R_i$ and $R_j$ are parallel and Lemma \ref{Lemma: CM06Prop2.7} implies $R_i \in \partial^2 \alpha_i$. In particular, $R_i \in \partial^2 \alpha_i \cap \partial^2 \alpha_j \cap \partial^2 \alpha_k$, which is a contradiction. This shows the first part of $(a)$.
	
	Note that $\proj_{R_i} R_j$ is a panel containing the panel $\proj_{R_i} P_j$ and hence $\proj_{R_i} R_j = \proj_{R_i} P_j$. As $P_i, P_j \in \partial \alpha_k$ and $P_i \subseteq R_i$, it follows from Lemma \ref{Lemma: stabilized implies parallel}$(b)$ that $\proj_{R_i} R_j = \proj_{R_i} P_j \in \partial \alpha_k$. Similarly, we obtain $\proj_{R_i} R_k \in \partial \alpha_j$. As $\{ \alpha_1, \alpha_2, \alpha_3 \}$ is a combinatorial triangle, we have in particular $\alpha_k \neq \pm \alpha_j$. Since for each panel there exists a unique reflection which stabilizes it, Lemma \ref{Lemma: CM06Prop2.7} implies that the panels are not parallel. This finishes the claim.
\end{proof}

\begin{lemma}\label{Lemma: reflection triangle is a chamber}
	Suppose $(W, S)$ is $2$-complete. If $T$ is a combinatorial triangle, then $(-\alpha, \beta) = \emptyset$ for all $\alpha \neq \beta \in T$. In particular, $\vert \alpha \cap \beta \cap R \vert = 1$ for each $R \in \partial^2 \alpha \cap \partial^2 \beta$.
\end{lemma}
\begin{proof}
	The first part is \cite[Proposition $2.3$]{Bi22}. Now let $R \in \partial^2 \alpha \cap \partial^2 \beta$. As roots are convex and $\alpha \neq \pm \beta$, we deduce $\alpha \cap \beta \cap R \neq \emptyset$. Assume $c, d \in \alpha \cap \beta \cap R$ with $c \neq d$. As roots and residues are convex, we can assume that $c$ and $d$ are contained in a panel $P \subseteq R$. Let $\gamma$ be a root with $P \in \partial \gamma$. Note that $R \in \partial^2 \gamma$. Using \cite[Lemma $2.9(b)$]{BiConstruction} we obtain $\epsilon \gamma \in (\alpha, \beta)$ for some $\epsilon \in \{+, -\}$. But then $P \subseteq \alpha \cap \beta \subseteq \epsilon \gamma$, which is a contradiction. This finishes the claim.
\end{proof}

\begin{lemma}\label{Lemma: A_2 tilde}
	Suppose $(W, S)$ is $2$-complete. Let $r, s, t, u \in S$ be such that $r \neq s\neq t \neq u$ and such that the panels $\P_r(1_W)$ and $\P_u(st)$ are parallel. Then $r=t$, $u=s$ and $m_{st} = 3$.
\end{lemma}
\begin{proof}
	By \cite[Lemma $2.15$]{AB08} and the fact that $m_{st} >2$ we have $\ell(rst) = 3$. As $\P_r(1_W)$ and $\P_u(st)$ are parallel, we must have $\ell(rstu) = 2$ by Lemma \ref{Lemma: results about parallel residues}$(b)$. But then Lemma \ref{Lemma: BiCoxGrowth2.8} implies $r=t$, $u=s$ and $m_{st} = 3$.
\end{proof}

\begin{remark}
	Suppose $(W, S)$ is $2$-complete and let $T = \{ \alpha_i, \alpha_j, \alpha_k \}$ be a combinatorial triangle. By definition and Lemma \ref{Lemma: 2-dimensional rank 2 residues not parallel}$(b)$ there exists a unique element contained in $\partial^2 \alpha_i \cap \partial^2 \alpha_j$.
\end{remark}

\begin{theorem}\label{Theorem: combinatorial triangle is chamber}
	Suppose $(W, S)$ is $2$-complete and $\tilde{A}_2$-free. Let $T = \{ \alpha_i, \alpha_j, \alpha_k \}$ be a combinatorial triangle. We denote the unique element contained in $\partial^2 \alpha_i \cap \partial^2 \alpha_j$ by $R_k$. Then $R_i \cap R_j \cap R_k$ is non-empty and contains a unique chamber.
\end{theorem}
\begin{proof}
	We first show that $R_i \cap R_j \cap R_k \neq \emptyset$. By Lemma \ref{Lemma: reflection triangle is a chamber}, $\alpha_i \cap \alpha_j \cap R_k$ contains a unique chamber which we denote by $c_k$. As $c_i \in R_i \subseteq \alpha_i$ and roots are convex, we deduce $\proj_{R_j} c_i \in \alpha_i$. Moreover, $c_i \in \alpha_k$ and hence $\proj_{R_j} c_i \in \alpha_k$. We deduce
	\[ \proj_{R_j} c_i \in \alpha_i \cap \alpha_k \cap R_j = \{ c_j \}. \]
	Define $P := \proj_{R_j} R_k$. As before, we have $c_j = \proj_{R_j} c_k \in P$ and Lemma \ref{Lemma: Equal projection chamber on residues} yields $c_j = \proj_{R_j} c_k = \proj_P c_k$. For $Q := \proj_{R_k} R_j$ we infer $c_k = \proj_Q c_j$ similarly. Without loss of generality we can assume $\ell(c_i, c_j) \geq \ell(c_j, c_k) \geq \ell(c_i, c_k)$. We define
	\begin{align*}
		&w_{ij} := \delta(c_i, c_j), &&w_{jk} := \delta(c_j, c_k), &&w_{ik} := \delta(c_i, c_k).
	\end{align*}
	Assume $\ell(w_{jk}) >0$ and hence $\ell(w_{ij}) \geq \ell(w_{jk}) > 0$. As $\proj_{R_i} R_j$ and $\proj_{R_j} R_i$ are panels by Lemma \ref{Lemma: combinatorial triangle yields triangle}$(a)$ which are parallel by Lemma \ref{Lemma: results about parallel residues}$(a)$, there exists a compatible path $(Q_0 := \proj_{R_i} R_j, \ldots, Q_q := \proj_{R_j} R_i)$ by Lemma \ref{Lemma: results about parallel panels}$(a)$. As $c_i \in Q_0$ and $c_j \in Q_q$, we have $q>0$. Using induction there exists a minimal gallery $(e_0 := c_i, \ldots, e_l := c_j)$ with $e_z \in R(Q_{j_z}, Q_{j_z +1})$ for all $0 \leq z \leq l$. Note that $c_j \in \proj_{R_j} R_i \cap \proj_{R_j} R_k$ and hence $\proj_{R_j} R_i \cap \proj_{R_j} R_k \neq \emptyset$.	
	
	Let $\proj_{R_j} R_i$ be an $s$-panel and let $\proj_{R_j} R_k$ be a $t$-panel. By Lemma \ref{Lemma: combinatorial triangle yields triangle}$(b)$ we have, in particular, $\proj_{R_j} R_i \neq \proj_{R_j} R_k$. As $\proj_{R_j} R_i \cap \proj_{R_j} R_k \neq \emptyset$, we have $s\neq t$. Let $\{ r, s \}$ be the type of $R(Q_{q-1}, Q_q)$. Then $\ell(w_{ij} sr) = \ell(w_{ij})$ (as $Q_{q-1}$ and $Q_q$ are opposite in $R(Q_{q-1}, Q_q)$) and by Lemma \ref{Lemma: BiCoxGrowth2.8}, $r$ is the unique element in $S$ with $\ell( w_{ij} r ) < \ell(w_{ij})$. As before, there exists a compatible path $(P_0 := \proj_{R_j} R_k, \ldots, P_p := \proj_{R_k} R_j)$ and we deduce $p>0$. Note that $\proj_{R_k} c_i = c_k \in \proj_{R_k} R_j$ and, hence, Lemma \ref{Lemma: Equal projection chamber on residues} yields $\proj_{P_p} c_i = \proj_{R_k} c_i = c_k$.
	
	If $r$ is not contained in the type of $R(P_0, P_1)$, then Lemma \ref{Lemma: BiCoxGrowth2.8} yields $\proj_{R(P_0, P_1)} c_i = \proj_{P_0} c_i$. As $\proj_{R_j} c_i = c_j \in \proj_{R_j} R_k$, Lemma \ref{Lemma: Equal projection chamber on residues} implies $c_j = \proj_{R_j} c_i = \proj_{P_0} c_i$. Now Lemma \ref{Lemma: no combinatorial triangle} yields
	\[ \ell(w_{ij}) = \ell(c_i, \proj_{P_0} c_i) < \ell(c_i, \proj_{P_p} c_i) = \ell(w_{ik}). \]	
	This is a contradiction. Thus we can assume that $R(P_0, P_1)$ is of type $\{r, t\}$. By Lemma \ref{Lemma: parallel panels - word become longer} we have $\ell( w_{ij} rt ) = \ell(w_{ij})$. If $m_{rt} \neq 3$, we deduce again a contradiction from Lemma \ref{Lemma: no combinatorial triangle}. Thus we can assume $m_{rt} = 3$. We distinguish the following cases:
	\begin{enumerate}
		\item[$p>1$:] Let $\{ r, x \}$ be the type of $R(P_1, P_2)$. Note that $(P_1, \ldots, P_p)$ is again a compatible path. If $\proj_{R(P_1, P_2)} c_i = \proj_{P_1} c_i$, then Lemma \ref{Lemma: no combinatorial triangle} yields again
		\[ \ell(w_{ij}) = \ell(c_i, \proj_{P_0} c_i) = \ell(c_i, \proj_{P_1} c_i) < \ell(c_i, \proj_{P_p} c_i) = \ell(w_{ik}) \]
		which is a contradiction. Thus we can assume $\proj_{R(P_1, P_2)} c_i \neq \proj_{P_1} c_i$ and hence $\ell( w_{ij} rtx ) = \ell(w_{ij}) -1$. This implies $\ell(w_{ij}rx) = \ell(w_{ij})-2$ and Lemma \ref{Lemma: parallel panels - word become longer} yields $x = s$. Moreover, we have $\ell(w_{ij}rst) = \ell(w_{ij})-3$. But then Lemma \ref{Lemma: parallel panels - word become longer} yields $m_{rs} = 3$. Since $c_j rs = \proj_{R(Q_{q-1}, Q_q)} c_i$, we have $\ell(w_{ij} rsrt) = \ell(w_{ij}) -2$ as above. Again, Lemma \ref{Lemma: parallel panels - word become longer} yields $\ell(w_{ij} rsts) = \ell(w_{ij}) -2$ and hence $m_{st} = 3$. But then $(\langle r, s, t \rangle, \{r, s, t\})$ is a subsystem of type $\tilde{A}_2$, which is a contradiction.
		
		\item[$p=1$:] Then we have $w_{jk} = rt$ and $c_k = c_j rt$. Moreover, we have $\ell(w_{jk}) = 2$ and $\ell(w_{ij}) = \ell(w_{ik})$, which implies $\ell(w_{ij}) = \ell(w_{jk}) = \ell(w_{ik}) = 2$ and, in particular, $w_{ij} = sr$, and $w_{ik} = st$. This implies $m_{rs} = 3$. Now it follows from Lemma \ref{Lemma: A_2 tilde} that $(\langle r, s, t \rangle, \{ r, s, t \})$ is a subsystem of type $\tilde{A}_2$, which is a contradiction.
	\end{enumerate}
	
	Thus we can assume $\ell(w_{jk}) = 0$. Then $\ell(w_{ik}) \leq \ell(w_{jk}) = 0$ and hence $c_i = c_k = c_j$. In particular, $R := R_i \cap R_j \cap R_k \neq \emptyset$ and hence $R$ is a residue. As $\proj_{R_i} R_j$ is a panel by Lemma \ref{Lemma: combinatorial triangle yields triangle}$(a)$, $R$ is either a panel or a single chamber. Suppose $R$ is a panel. Then $\proj_{R_i} R_j = R = \proj_{R_i} R_k$ and, in particular, $\proj_{R_i} R_j$ and $\proj_{R_i} R_k$ are parallel. But this is a contradiction to Lemma \ref{Lemma: combinatorial triangle yields triangle}$(b)$. Hence $R$ is a single chamber and we are done.
\end{proof}

\begin{remark}
	Note that the previous theorem becomes false, if $(W, S)$ is not $\tilde{A}_2$-free.
\end{remark}

\section{Triangles in buildings}\label{Section: Triangles in buildings}

In this section we generalize the notion of combinatorial triangles in $\Sigma(W, S)$ (i.e.\ in apartments) to triangles in building. We remark that in the literature a combinatorial triangle is sometimes also called triangle. The main goal of this section is Theorem \ref{Theorem: 2-complete and A_2 tilde free triangle contains unique chamber}, where we prove a generalization of Theorem \ref{Theorem: combinatorial triangle is chamber}.

\begin{convention}
	In this section we let $\Delta = (\C, \delta)$ be a building of type $(W, S)$.
\end{convention}

\begin{definition}
	A set $T$ of three residues of spherical type and of rank $2$ is called \emph{triangle}, if the following hold:
	\begin{enumerate}[label=(T\arabic*)]
		\item For $R, Q \in T$ with $R\neq Q$, the residue $\proj_R Q$ is a panel.
		
		\item The panels $\proj_R P$ and $\proj_R Q$ are not parallel, where $T = \{P, Q, R\}$.
	\end{enumerate}
	Let $\Sigma$ be an apartment of $\Delta$. A triangle $T$ is called \emph{$\Sigma$-triangle} if $\Sigma \cap R \neq \emptyset$ for all $R \in T$.
\end{definition}

We first show that triangles exist. The next lemma shows that any reflection triangle in an apartment canonically provides a triangle in a building.

\begin{lemma}\label{Lemma: construction of triangles}
	Let $\Sigma$ be an apartment and let $\beta_1, \beta_2, \beta_3 \in \Phi^{\Sigma}$ be three roots such that $\{ r_{\beta_1}, r_{\beta_2}, r_{\beta_3} \}$ is a reflection triangle. Let $T := \{ R_1, R_2, R_3 \}$ be a set of three spherical residues of rank $2$ with the following properties:
	\begin{enumerate}[label=(\roman*)]
		\item $\Sigma \cap R \neq \emptyset$ for all $R \in T$;
		
		\item $\Sigma \cap R_k \in \partial^2 \alpha_i \cap \partial^2 \alpha_j$ for all possibilities $\{ i, j, k \} = \{1, 2, 3\}$.
	\end{enumerate}
	Then $T$ is a $\Sigma$-triangle.
\end{lemma}
\begin{proof}
	We abbreviate $\Sigma_i := \Sigma \cap R_i$ for $i \in \{1, 2, 3\}$. By Remark \ref{Remark: reflection triangle plus orientation yields triangle}, there exists $\gamma_i \in \{ \beta_i, -\beta_i \}$ for all $i \in \{1, 2, 3\}$ such that $\{ \gamma_1, \gamma_2, \gamma_3 \}$ is a combinatorial triangle. Using Lemma \ref{Lemma: combinatorial triangle yields triangle} and Lemma \ref{Lemma: residues and apartments}$(a)$ we obtain the following:
	\begin{enumerate}[label=(\alph*)]
		\item $\Sigma \cap \proj_{R_i} R_j = \proj_{\Sigma_i} \Sigma_j$ is a panel in $\Sigma$ for $i\neq j \in \{1, 2, 3\}$.
		
		\item The panels $\Sigma \cap \proj_{R_k} R_i = \proj_{\Sigma_k} \Sigma_i$ and $\Sigma \cap \proj_{R_k} R_j = \proj_{\Sigma_k} \Sigma_j$ are not parallel in $\Sigma$ for $\{ i, j, k \} = \{ 1, 2, 3 \}$.
	\end{enumerate}
	
	Lemma \ref{Lemma: residues and apartments}$(b)$ yields that $\proj_{R_i} R_j$ is a panel for all $i \neq j \in \{ 1, 2, 3 \}$. Now Lemma \ref{Lemma: results about parallel residues}$(c)$ implies that $\proj_{R_k} R_i$ and $\proj_{R_k} R_j$ are not parallel, where $\{ i, j, k \} = \{ 1, 2, 3 \}$. By definition, $T$ is a $\Sigma$-triangle.
\end{proof}

\begin{remark}
	Let $\Sigma$ be an apartment of $\Delta$ and let $T = \{ R_i, R_j, R_k \}$ be a $\Sigma$-triangle. Then $\proj_{R_i} R_j$ and $\proj_{R_j} R_i$ are panels by assumption and $\Sigma \cap \proj_{R_i} R_j$ and $\Sigma \cap \proj_{R_j} R_i$ are panels in $\Sigma$ by Lemma \ref{Lemma: residues and apartments}$(b)$. By Lemma \ref{Lemma: results about parallel residues}$(a)$ and $(c)$, the panels $\Sigma \cap \proj_{R_i} R_j$ and $\Sigma \cap \proj_{R_j} R_i$ are parallel in $\Sigma$. It follows from Lemma \ref{Lemma: CM06Prop2.7} that there exists (up to sign) a unique root $\beta_k \in \Phi^{\Sigma}$ with $\Sigma \cap \proj_{R_i} R_j, \Sigma \cap \proj_{R_j} R_i \in \partial \beta_k$.
\end{remark}

\begin{lemma}\label{Lemma: triangle yields reflection triangle}
	Let $\Sigma$ be an apartment and let $T = \{ R_i, R_j, R_k \}$ be a $\Sigma$-triangle. Then the following hold:
	\begin{enumerate}[label=(\alph*)]
		\item Let $\beta_k \in \Phi^{\Sigma}$ be one of the two roots with $\Sigma \cap \proj_{R_i} R_j, \Sigma \cap \proj_{R_j} R_i \in \partial \beta_k$. Then $\{ r_{\beta_i}, r_{\beta_j}, r_{\beta_k} \}$ is a reflection triangle. Moreover, there exists $\gamma_l \in \{ \beta_l, -\beta_l \}$ with $\Sigma \cap R_l \subseteq \gamma_l$ for all $l \in \{ i, j, k \}$, and $\{ \gamma_i, \gamma_j, \gamma_k \}$ is a combinatorial triangle.
		
		\item If $(W, S)$ is $2$-complete and $\tilde{A}_2$-free, then $R_i \cap R_j \cap R_k$ is non-empty and contains a unique chamber.
		
		\item If $(W, S)$ is $2$-complete, then $\proj_{R_i} R_j \cap \proj_{R_i} R_k$ is non-empty and contains a unique chamber.
	\end{enumerate}
\end{lemma}
\begin{proof}
	For all $a\neq b \in \{i, j, k\}$ we abbreviate the following:
	\allowdisplaybreaks
	\begin{align*}
		& \Sigma_a := \Sigma \cap R_a, && P_a^b := \proj_{R_a} R_b, && \Sigma_a^b := \Sigma \cap P_a^b.
	\end{align*}
	As $\Sigma_k^j \in \partial \beta_i$ and $\Sigma_k^j \subseteq \Sigma_k$, we have $\Sigma_k \in \partial^2 \beta_i$. Using similar arguments, we deduce $\Sigma_k \in \partial^2 \beta_i \cap \partial^2 \beta_j$, which implies $o(r_{\beta_i} r_{\beta_j}) < \infty$. Next we show $\partial^2 \beta_i \cap \partial^2 \beta_j \cap \partial^2 \beta_k = \emptyset$. Assume by contrary that there exists $ Q \in \partial^2 \beta_i \cap \partial^2 \beta_j \cap \partial^2 \beta_k$. Then $Q$ and $\Sigma_i$ (resp.\ $Q$ and $\Sigma_j$) are parallel by \cite[Lemma $2.5(a)$]{BiCoxGrowth}. Now \cite[Corollary $21.21$]{MPW15} implies that $\Sigma_i$ and $\Sigma_j$ are parallel. Using Lemma \ref{Lemma: results about parallel residues}$(c)$, the residues $R_i$ and $R_j$ are parallel as well. But $\proj_{R_i} R_j$ is a panel, hence $\proj_{R_i} R_j \neq R_i$, which is a contradiction. We conclude $\partial^2 \beta_i \cap \partial^2 \beta_j \cap \partial^2 \beta_k = \emptyset$ and $\{ r_{\beta_i}, r_{\beta_j}, r_{\beta_k} \}$ is a reflection triangle. The second part of $(a)$ follows from Remark \ref{Remark: reflection triangle plus orientation yields triangle}.
	
	For Assertion $(b)$ we note that $\Sigma_i \cap \Sigma_j \cap \Sigma_k$ is non-empty and contains a unique chamber by $(a)$ and Theorem \ref{Theorem: combinatorial triangle is chamber}. This implies $R := R_i \cap R_j \cap R_k \neq \emptyset$ and hence $R$ is a residue. As $\proj_{R_i} R_j$ is a panel, $R$ is either a panel or a single chamber. Suppose $R$ is a panel. Then $\proj_{R_i} R_j = R = \proj_{R_i} R_k$ and, hence, $\proj_{R_i} R_j$ and $\proj_{R_i} R_k$ are parallel. This is a contradiction to the fact that $T$ is a triangle. Hence $R$ is a single chamber and we are done.
	
	For assertion $(c)$ we first note that $\vert \gamma_i \cap \gamma_j \cap \Sigma_k \vert = 1$ holds by Lemma \ref{Lemma: reflection triangle is a chamber}. We denote the unique chamber contained in $\gamma_i \cap \gamma_j \cap \Sigma_k$ by $c_k$. As roots are convex, we have $\proj_{R_i} c_j, \proj_{R_i} c_k \in \gamma_j \cap \gamma_k \cap \Sigma_i = \{ c_i \}$. This implies $P_i^j \cap P_i^k \neq \emptyset$. Assume that $P_i^j \cap P_i^k$ is not a chamber. As both are panels, we must have $P_i^j = P_i^k$. But then $P_i^j$ and $P_i^k$ would be parallel, which is a contradiction. This finishes the claim.
\end{proof}

\begin{convention}\label{Convention: triangle}
	Let $T = \{R_1, R_2, R_3\}$ be a triangle. For $i\neq j \in \{1, 2, 3\}$ we let $P_i^j := \proj_{R_i} R_j$. Note that $P_i^j$ is a panel by definition. By Lemma \ref{Lemma: results about parallel residues}$(a)$, the panels $P_i^j$ and $P_j^i$ are parallel. By Lemma \ref{Lemma: results about parallel panels}$(a)$ there exists a compatible path $(P_0 := P_2^1, \ldots, P_n := P_1^2)$ between $P_2^1$ and $P_1^2$, and a compatible path $(Q_0 := P_3^2, \ldots, Q_m := P_2^3)$ between $P_3^2$ and $P_2^3$. 
	
	We extend the compatible path $(P_0, \ldots, P_n)$ by two panels. Let $P_{n+1}$ be any panel contained in $R_1$ which is opposite to $P_n$ in $R_1$. Note that $R(P_n, P_{n+1}) = R_1$. As $\proj_{R_1} P_0 = P_n$, the path $(P_0, \ldots, P_{n+1})$ is again a compatible path. Similarly, we obtain a compatible path $(P_{-1}, P_0, \ldots, P_{n+1})$ with $R(P_{-1}, P_0) = R_2$, as well as a compatible path $(Q_{-1}, \ldots, Q_{m+1})$ with $R(Q_{-1}, Q_0) = R_3$ and $R(Q_m, Q_{m+1}) = R_2$.
\end{convention}

\subsection*{A technical result} 

Assume that $(W, S)$ is $2$-complete and let $T = \{ R_1, R_2, R_3 \}$ be a triangle, which is not a $\Sigma$-triangle for any apartment $\Sigma$. Let $c\in P_3^1$ and let $d := \proj_{P_2^1} c$. Then $\proj_{R_1} d \neq d$, as $T$ is not a $\Sigma$-triangle. Let $(d_0 := d, \ldots, d_k := \proj_{R_1} d)$ be a minimal gallery such that for all $0 \leq i \leq k-1$ there exists $0 \leq j_i \leq n-1$ with $\{ d_i, d_{i+1} \} \subseteq R(P_{j_i}, P_{j_i +1})$. Let $0 \leq z \leq k-1$ be maximal such that there exists an apartment containing $c, d_0, \ldots, d_z$, and let $\Sigma$ be such an apartment. Let $\alpha \in \Phi^{\Sigma}$ be the root containing $d_z$ but not $P \cap \Sigma$, where $P$ is the panel containing $d_z$ and $d_{z+1}$. Note that $\Sigma \cap P_2^1 \neq \emptyset$ and $\Sigma \cap P_2^3 \neq \emptyset$ (cf.\ Lemma \ref{Lemma: residues and apartments}). Let $\gamma_3 \in \Phi^{\Sigma}$ be the root containing $d$ but not $\Sigma \cap P_2^1$ and let $\gamma_1 \in \Phi^{\Sigma}$ be the root containing $d$ with $\Sigma \cap P_2^3 \in \partial \gamma_1$.

\begin{proposition}\label{Proposition: auxiliary result}
	The following hold:
	\begin{enumerate}[label=(\alph*)]
		\item There exists a unique $0 \leq j \leq m$ such that $\Sigma \cap R(Q_{j-1}, Q_j) \in \partial^2 \alpha$.
		
		\item $\{ R(P_{j_z}, P_{j_z +1}), R(Q_{j-1}, Q_j), R_2 \}$ is a $\Sigma$-triangle.
		
		\item $\{ r_{\alpha}, r_{\gamma_1}, r_{\gamma_3} \}$ is a reflection triangle.
	\end{enumerate}
\end{proposition}
\begin{proof}
	We abbreviate $R_P := R(P_{j_z}, P_{j_z +1})$. As $\Sigma \cap R_P \in \partial^2 \alpha \cap \partial^2 \gamma_3$ by Lemma \ref{Lemma: CM06Prop2.7}, we have $o(r_{\alpha} r_{\gamma_3}) < \infty$. Moreover, we have $\Sigma \cap R_2 \in \partial^2 \gamma_1 \cap \partial^2 \gamma_3$ and hence $o(r_{\gamma_1} r_{\gamma_3}) < \infty$. Let $f \in P_2^3 \cap \gamma_1$ be the unique chamber. Then $\proj_{R_3} f \in P_3^2$ and as $\Sigma \cap R_3 \neq \emptyset$, we have $\proj_{R_3} f \in P_3^2 \cap \Sigma$. As $P_2^3$ and $P_3^2$ are parallel, Lemma \ref{Lemma: results about parallel residues}$(c)$ yields that $\Sigma \cap P_2^3$ and $\Sigma \cap P_3^2$ are parallel. As $\Sigma \cap P_2^3 \in \partial \gamma_1$, Lemma \ref{Lemma: CM06Prop2.7} implies $\Sigma \cap P_3^2 \in \partial \gamma_1$ and hence $\proj_{R_3} f \in \gamma_1$ (as roots are convex). Note that $\{ \Sigma \cap R(Q_{j-1}, Q_j) \mid 0 \leq j \leq m \} \subseteq \partial^2 \gamma_1$. We distinguish the following two cases:
	\begin{enumerate}[label=(\roman*)]
		\item\label{Case i} $\proj_{R_3} f \in \alpha$: Then, as $c\notin \alpha$ by Lemma \ref{Lemma: no apartment contains chambers}, we have $\Sigma \cap R_3 \in \partial^2 \gamma_1 \cap \partial^2 \alpha \neq \emptyset$. In particular, we deduce $o(r_{\alpha} r_{\gamma_1}) < \infty$.
		
		\item\label{Case ii} $\proj_{R_3} f \in (-\alpha)$: By Lemma \ref{Lemma: triangle projection} we have $\proj_{R_2} d_k = d_0$. This implies that we can extend $(d_{z+1}, \ldots, d_0)$ to a minimal gallery from $d_{z+1}$ to $f$. Replacing $d_{z+1}$ by the unique chamber contained in $\Sigma$ which is $\delta(d_z, d_{z+1})$-adjacent to $d_z$, we infer $f\in \alpha$, as a minimal gallery crosses each wall at most once. As $\proj_{R_3} f \in (-\alpha)$, there exists $1 \leq j \leq m$ with $Q_{j-1} \cap \Sigma \subseteq (-\alpha)$ and $Q_j \cap \Sigma \subseteq \alpha$. Then $R(Q_{j-1}, Q_j) \cap \Sigma \in \partial^2 \gamma_1 \cap \partial^2 \alpha$ and hence $o(r_{\gamma_1} r_{\alpha}) < \infty$.
	\end{enumerate}
	
	Note that $\partial^2 \gamma_1 \cap \partial^2 \gamma_3 = \{ \Sigma \cap R_2 \}$ and $\partial^2 \alpha \cap \partial^2 \gamma_3 = \{ \Sigma \cap R_P \}$ by Lemma \ref{Lemma: 2-dimensional rank 2 residues not parallel}$(b)$. Assume that $\Sigma \cap R_2 \in \partial^2 \alpha$. Then we would have $\Sigma \cap R_P = \Sigma \cap R_2$ and hence $d_z, d_{z+1} \in R_P = R_2$. This implies $\proj_{R_2} \left( \proj_{R_1} d \right) \neq d$, which is a contradiction to Lemma \ref{Lemma: triangle projection}. We infer $\partial^2 \gamma_3 \cap \partial^2 \gamma_1 \cap \partial^2 \alpha = \emptyset$ and hence $\{ r_{\gamma_1}, r_{\gamma_3}, r_{\alpha} \}$ is a reflection triangle. Recall that $\{ \Sigma \cap R(Q_{j-1}, Q_j) \mid 0 \leq j \leq m \} \subseteq \partial^2 \gamma_1$ and we know by \ref{Case i} and \ref{Case ii} that $\Sigma \cap R(Q_{j-1}, Q_j) \in \partial^2 \alpha$ for some $0 \leq j \leq m$. The uniqueness of $j$ follows from the fact that $\partial^2 \gamma_1 \cap \partial^2 \alpha$ contains exactly one element by Lemma \ref{Lemma: 2-dimensional rank 2 residues not parallel}$(b)$. Let $R_Q \in \{ R(Q_{j-1}, Q_j) \mid 0 \leq j \leq m \}$ with $\partial^2 \alpha \cap \partial^2 \gamma_1 = \{ \Sigma \cap R_Q \}$. Then Lemma \ref{Lemma: construction of triangles} implies that $\{ R_2, R_P, R_Q \}$ is a $\Sigma$-triangle.
\end{proof}

\subsection*{Two generalizations}

Theorem \ref{Theorem: intersection of two residues is a chamber} can be seen as a generalization of Lemma \ref{Lemma: reflection triangle is a chamber}.

\begin{theorem}\label{Theorem: intersection of two residues is a chamber}
	Suppose that $(W, S)$ is $2$-complete and let $T = \{ R_1, R_2, R_3 \}$ be a triangle. Then $\proj_{R_2} R_1 \cap \proj_{R_2} R_3$ is non-empty and contains a unique chamber.
\end{theorem}
\begin{proof}
	If $T$ is a $\Sigma$-triangle for some apartment $\Sigma$, then the claim follows from Lemma \ref{Lemma: triangle yields reflection triangle}$(c)$. Thus we can assume that $T$ is not a $\Sigma$-triangle for any apartment $\Sigma$. As $T$ is a triangle, we have $P_2^1 \neq P_2^3$ by axiom (T$2$). Thus it suffices to show that $P_2^1 \cap P_2^3 \neq \emptyset$. Let $c\in P_3^1$ be any chamber and let $d := \proj_{P_2^1} c$. As $T$ is not a $\Sigma$-triangle, we have $d \neq \proj_{R_1} d$. Note that $\proj_{R_1} d = \proj_{P_1^2} d$ holds by Lemma \ref{Lemma: Equal projection chamber on residues}. Using induction, there exists a minimal gallery $(d_0 := d, \ldots, d_k := \proj_{R_1} d)$ such that for all $0 \leq i \leq k-1$ there exists $0 \leq j_i \leq n-1$ with $\{ d_i, d_{i+1} \} \subseteq R(P_{j_i}, P_{j_i +1})$.
	
	Now let $0 \leq z \leq k$ be maximal such that there exists an apartment containing $c, d=d_0, \ldots, d_z$ and let $\Sigma$ be such an apartment. As $T$ is not a $\Sigma$-triangle, we have $z<k$. Define $R_P := R(P_{j_z}, P_{j_z +1})$. Let $\alpha \in \Phi^{\Sigma}$ be the root containing $d_z$ but not $\Sigma \cap P$, where $P$ is the panel containing $d_z$ and $d_{z+1}$. Let $0 \leq j \leq m$ be the unique element as in Proposition \ref{Proposition: auxiliary result}$(a)$ and define $R_Q := R(Q_{j-1}, Q_j)$. Then $\{ R_P, R_Q, R_2 \}$ is a $\Sigma$-triangle by Proposition \ref{Proposition: auxiliary result}$(b)$. Applying Lemma \ref{Lemma: triangle yields reflection triangle}$(c)$ we deduce that $\proj_{R_2} R_P \cap \proj_{R_2} R_Q$ is non-empty and contains a unique chamber. Now Lemma \ref{Lemma: Projection on compatible paths} implies $\proj_{R_2} R_Q = \proj_{R_2} R_3$ and $\proj_{R_2} R_P = \proj_{R_2} R_1$ (as $(P_{n+1}, \ldots, P_{-1})$ is a compatible path).
\end{proof}

\begin{theorem}\label{Theorem: 2-complete and A_2 tilde free triangle contains unique chamber}
	Suppose that $(W, S)$ is $2$-complete and $\tilde{A}_2$-free, and suppose that $T = \{ R_1, R_2, R_3 \}$ is a triangle. Then $R_1 \cap R_2 \cap R_3$ is non-empty and contains a unique chamber.
\end{theorem}
\begin{proof}
	We let $\{i, j, k\} = \{1, 2, 3\}$. We know by Theorem \ref{Theorem: intersection of two residues is a chamber} that $P_i^j \cap P_i^k$ is non-empty and contains a unique chamber, which we will denote by $c_i$. Note that $\proj_{R_j} c_i = \proj_{P_j^i} c_i$ by Lemma \ref{Lemma: Equal projection chamber on residues}. As $c_j \in P_j^i$, we have $\ell(\proj_{P_j^i} c_i, c_j) \leq 1$. By definition we have $P_j^k \neq P_j^i$ and hence $\proj_{P_j^k} c_i = c_j$. Without loss of generality we can assume $\ell(c_3, c_2) \geq \ell(c_2, c_1) \geq \ell(c_3, c_1)$. We define
	\begin{align*}
		&w_{32} := \delta(c_3, c_2), &&w_{21} := \delta(c_2, c_1), &&w_{31} := \delta(c_3, c_1).
	\end{align*}
	
	If $T$ is a $\Sigma$-triangle for some apartment $\Sigma$, then the claim follows from Lemma \ref{Lemma: triangle yields reflection triangle}$(b)$. Thus we can assume that $T$ is not a $\Sigma$-triangle for any apartment $\Sigma$. In particular, we have $\proj_{R_1} c_2 \neq c_2$ and $n>0$ (cf.\ Convention \ref{Convention: triangle}). Using induction, there exists a minimal gallery $(d_0 := c_2, \ldots, d_l := \proj_{R_1} c_2)$ such that for all $0 \leq x \leq l-1$ there exists $0 \leq j_x \leq n-1$ with $\{ d_x, d_{x+1} \} \subseteq R(P_{j_x}, P_{j_x +1})$. Let $0 \leq z \leq l$ be maximal such that there exists an apartment containing $c_3, d_0, \ldots, d_z$ and let $\Sigma$ be such an apartment. As $T$ is not a $\Sigma$-triangle, we deduce $z<l$. Let $0 \leq j \leq m$ be the unique element as in Proposition \ref{Proposition: auxiliary result}$(a)$. We define $R_P := R(P_{j_z}, P_{j_z +1})$ and $R_Q := R(Q_{j-1}, Q_j)$. Then $\{ R_P, R_Q, R_2 \}$ is a $\Sigma$-triangle by Proposition \ref{Proposition: auxiliary result}$(b)$. In particular, $R_P \cap R_2 \cap R_Q$ is non-empty and contains a unique chamber by Lemma \ref{Lemma: triangle yields reflection triangle}$(b)$.
	
	\emph{Claim: $R_P \cap R_2 \cap R_Q = \{d_0\}$.}
	
	By Lemma \ref{Lemma: Projection on compatible paths} we have $\proj_{R_2} R_P = \proj_{R_2} R_1 = P_2^1$ and $\proj_{R_2} R_Q = \proj_{R_2} R_3 = P_2^3$. As $R_2 \cap R_P \neq \emptyset$, the residue $R_2 \cap R_P$ is either a chamber, or a panel. If $R_2 \cap R_P$ would be a chamber, then \cite[Lemma $2.15$]{AB08} implies that $\proj_{R_2} R_P$ is also just a chamber. This is a contradiction and hence $R_2 \cap R_P$ is a panel. We infer $R_2 \cap R_P = \proj_{R_2} R_P$. Using the same arguments, we deduce $R_2 \cap R_Q = \proj_{R_2} R_Q$. We infer
	\[ R_2 \cap R_P \cap R_Q = ( \proj_{R_2} R_P ) \cap ( \proj_{R_2} R_Q ) = P_2^1 \cap P_2^3 = \{ d_0 \}. \]
	
	Let $\gamma_3 \in \Phi^{\Sigma}$ be the root containing $d_0$ but not $\Sigma \cap P_2^1$ and let $\alpha \in \Phi^{\Sigma}$ be the root containing $d_z$ but not $\Sigma \cap P$, where $P$ is the panel containing $d_z$ and $d_{z+1}$. Then $d_0, \ldots, d_z \in \gamma_3 \cap \alpha \cap R_P$. Let $\gamma_1 \in \Phi^{\Sigma}$ be the root containing $d_0$ with $\Sigma \cap P_2^3 \in \partial \gamma_1$. By Proposition \ref{Proposition: auxiliary result}$(c)$, we obtain that $\{ r_{\alpha}, r_{\gamma_1}, r_{\gamma_3} \}$ is a reflection triangle.
	
	As $\Sigma \cap R_2 \in \partial^2 \gamma_1 \cap \partial^2 \gamma_3$ and $d_0 \in \alpha$, we deduce $\Sigma \cap R_2 \subseteq \alpha$. Moreover, $\Sigma \cap R_Q \subseteq \gamma_3$ and $\Sigma \cap R_P \subseteq \gamma_1$. This implies that $\{ \alpha, \gamma_1, \gamma_3 \}$ is a combinatorial triangle. It follows from Lemma \ref{Lemma: reflection triangle is a chamber} that $z=0$. Suppose $\delta(c_3, d_1) = \delta(c_3, d_0) \delta(d_0, d_1)$. Then there would exist an apartment containing $c_3, d_0, d_1$, which is a contradiction to the maximality of $z$. Thus we have $\delta(c_3, d_0) = \delta(c_3, d_1)$ and hence $\ell(\delta(c_3, d_0) \delta(d_0, d_1) ) = \ell(c_3, d_0) -1$. Recall that $\ell(\proj_{R_2} c_3, c_2) \leq 1$ and that $\proj_{P_0} c_3 = \proj_{P_2^1} c_3 = c_2$ and $\proj_{P_n} c_3 = \proj_{P_1^2} c_3 = c_1$. We distinguish the following cases:
	\begin{enumerate}[label=(\alph*)]
		\item $\ell(\proj_{R_2} c_3, c_2) = 0$: As $\ell(\delta(c_3, d_0) \delta(d_0, d_1) ) = \ell(c_3, d_0) -1$, we infer from Lemma \ref{Lemma: parallel panels - word become longer} that $\ell(\proj_{R(P_0, P_1)} c_3, c_2) = 1$. 
		
		\item $\ell(\proj_{R_2} c_3, c_2) = 1$: Then Lemma \ref{Lemma: BiCoxGrowth2.8} implies similarly as in the previous case that $\ell(\proj_{R(P_0, P_1)} c_3, c_2) = 1$.
	\end{enumerate}
	As $P_0$ and $P_1$ are opposite in $R(P_0, P_1)$ and $\delta(c_3, d_0) = \delta(c_3, d_1)$, we deduce that $\ell(c_3, \proj_{P_0} c_3) = \ell(c_3, c_2) < \ell(c_3, \proj_{P_1} c_3)$ and $\ell(\proj_{R(P_0, P_1)} c_3, \proj_{P_1} c_3) \geq 2$ hold. Using Lemma \ref{Lemma: no combinatorial triangle}, we deduce
	\[ \ell(w_{32}) = \ell(c_3, c_2) = \ell(c_3, \proj_{P_0} c_3) < \ell(c_3, \proj_{P_n} c_3) = \ell(c_3, c_1) = \ell(w_{31}). \]
	But this is a contradiction to the assumption $\ell(c_3, c_2) \geq \ell(c_3, c_1)$.
\end{proof}

\section{Groups of Kac-Moody type}\label{Section: Isomorphisms of RGD-systems}

\subsection*{Twin buildings}

Let $\Delta_+ = (\C_+, \delta_+)$ and $\Delta_- = (\C_-, \delta_-)$ be two buildings of the same type $(W, S)$. A \textit{codistance} (or \textit{twinning}) between $\Delta_+$ and $\Delta_-$ is a mapping $\delta_*: (\C_+ \times \C_-) \cup (\C_- \times \C_+) \to W$ satisfying the following axioms, where $\epsilon \in \{+,-\}$, $x\in \Cp$, $y\in \Cm$ and $w = \delta_*(x, y)$:
\begin{enumerate}[label=(Tw\arabic*)]
	\item $\delta_*(y, x) = w^{-1}$;
	
	\item if $z\in \Cm$ is such that $s := \delta_{-\epsilon}(y, z) \in S$ and $\ell(ws) = \ell(w) -1$, then $\delta_*(x, z) = ws$;
	
	\item if $s\in S$, there exists $z\in \C_{-\epsilon}$ such that $\delta_{-\epsilon}(y, z) = s$ and $\delta_*(x, z) = ws$.
\end{enumerate}
A \textit{twin building of type $(W, S)$} is a triple $\Delta = (\Delta_+, \Delta_-, \delta_*)$, where $\Delta_+ = (\C_+, \delta_+)$ and $\Delta_- = (\C_-, \delta_-)$ are buildings of type $(W, S)$ and where $\delta_*$ is a twinning between $\Delta_+$ and $\Delta_-$. Two chambers $c_+ \in \C_+$ and $c_- \in \C_-$ are called \emph{opposite} if $\delta_*(c_+, c_-) = 1_W$. The twin building is called \textit{thick}, if $\Delta_+$ and $\Delta_-$ are thick. A \textit{panel/residue} of a twin building is a panel/residue of one of its two buildings. Two residues $R$ of $\Delta_+$ and $T$ of $\Delta_-$ are called \emph{opposite} if they have the same type and if there exists $r\in R$ and $t\in T$ which are opposite.

Let $\Sigma_+ \subseteq \C_+$ and $\Sigma_- \subseteq \C_-$ be apartments of $\Delta_+$ and $\Delta_-$, respectively. Then the set $\Sigma := \Sigma_+ \cup \Sigma_-$ is called a \textit{twin apartment} if for all $\epsilon \in \{+, -\}$ and $x \in \Sigma_{\epsilon}$ there exists exactly one $y\in \Sigma_{-\epsilon}$ with $\delta_*(x, y) = 1_W$. Moreover, for any $c_+ \in \C_+$ and $c_- \in \C_-$ with $\delta_*(c_+, c_-) = 1_W$ there exists a unique twin apartment containing $c_+$ and $c_-$ (cf.\ \cite[Proposition $5.179(1)$]{AB08}). We denote this unique twin apartment by $\Sigma(c_+, c_-)$. An \emph{automorphism} of $\Delta$ is a bijection $\phi: \C_+ \cup \C_- \to \C_+ \cup \C_-$ which preserves the sign, the distances $\delta_+$ and $\delta_-$ and the codistance $\delta_*$.

Let $\Delta = (\Delta_+, \Delta_-, \delta_*)$ be a twin building of type $(W, S)$. A \emph{twin root of $\Delta$} is the convex hull of a pair of chambers at codistance $1$, i.e.\ a pair $\{x, y\}$ such that $\delta_*(x, y) \in S$. For more information we refer to \cite[Section $5.8.5$]{AB08}. A twin root can be seen as the union of two roots $\alpha_+ \cup \alpha_-$, where $\alpha_{\epsilon} \subseteq \Sigma_{\epsilon}$ is a root in an apartment $\Sigma_{\epsilon}$ for $\epsilon \in \{+, -\}$ and $\Sigma_+ \cup \Sigma_-$ is a twin apartment. Moreover, if we fix a twin apartment $\Sigma$, then the set of twin roots in $\Sigma$ is in one-to-one correspondence with the set of roots in one half of the twin apartment. Let $\Sigma$ be any twin apartment and let $R$ be a spherical residue with $\Sigma \cap R \neq \emptyset$. Then we denote by $\Phi^{\Sigma}$ the set of all twin roots of $\Sigma$ and we let $\Phi^{\Sigma}(R)$ be the set of all twin roots $\beta \in \Phi^{\Sigma}$ such that $R \cap \beta$ and $R \cap (-\beta)$ are both non-empty.

\subsection*{Root group data}

An \emph{RGD-system of type $(W, S)$} is a pair $\mathcal{D} = \left( G, \left( U_{\alpha} \right)_{\alpha \in \Phi}\right)$ consisting of a group $G$ together with a family of subgroups $U_{\alpha}$ (called \emph{root groups}) indexed by the set of roots $\Phi$, which satisfies the following axioms, where $H := \bigcap_{\alpha \in \Phi} N_G(U_{\alpha})$ and $U_{\epsilon} := \langle U_{\alpha} \mid \alpha \in \Phi_{\epsilon} \rangle$ for $\epsilon \in \{+, -\}$:
\begin{enumerate}[label=(RGD\arabic*)] \setcounter{enumi}{-1}
	\item For each $\alpha \in \Phi$, we have $U_{\alpha} \neq \{1\}$.
	
	\item For each prenilpotent pair $\{ \alpha, \beta \} \subseteq \Phi$ with $\alpha \neq \beta$, the commutator group $[U_{\alpha}, U_{\beta}]$ is contained in the group $U_{(\alpha, \beta)} := \langle U_{\gamma} \mid \gamma \in (\alpha, \beta) \rangle$.
	
	\item For each $s\in S$ and each $u\in U_{\alpha_s} \backslash \{1\}$, there exist $u', u'' \in U_{-\alpha_s}$ such that the product $m(u) := u' u u''$ conjugates $U_{\beta}$ onto $U_{s\beta}$ for each $\beta \in \Phi$.
	
	\item For each $s\in S$, the group $U_{-\alpha_s}$ is not contained in $U_+$.
	
	\item $G = H \langle U_{\alpha} \mid \alpha \in \Phi \rangle$.
\end{enumerate}

It is well-known that $\mathcal{D}$ acts on a twin building, which is denoted by $\Delta(\mathcal{D})$ (cf.\ \cite[Section $8.9$]{AB08}). This twin building is a so-called \emph{Moufang twin building} (cf.\ \cite[Section $8.3$]{AB08}). There is a distinguished pair of opposite chambers in $\Delta(\mathcal{D})$ corresponding to $B_{\epsilon} := HU_{\epsilon}$ for $\epsilon \in \{+, -\}$. We refer to this pair as the \emph{fundamental pair of opposite chambers} and we refer to the unique twin apartment containing these two chambers as the \emph{fundamental twin apartment}.

\subsection*{Two results about Moufang buildings}

For more information about \emph{Moufang buildings} we refer to \cite[Section $7.3$]{AB08}.

\begin{lemma}\label{Lemma: no panel is stabilized by all root groups}
	Let $\Delta = (\C, \delta)$ be a Moufang building of rank $2$ and of irreducible and spherical type. Let $\Sigma$ be an apartment and let $c_+, c_- \in \Sigma$ be the two fundamental opposite chambers. Let $(U_{\alpha})_{\alpha \in \Phi^{\Sigma}}$ be the family of corresponding root groups. For $s\in S$ we let $Q_s$ be a panel with $\langle U_{\alpha_s} \cup U_{-\alpha_s} \rangle \leq \Stab(Q_s)$. Then we have $Q_s \neq Q_t$.
	
	In particular, if $Q_s$ and $Q_t$ are parallel, then $Q_s$ and $Q_t$ are opposite.
\end{lemma}
\begin{proof}
	We assume by contrary that $Q_s = Q_t$. As $\P_s(c_+)$ is not stabilized by $U_{-\alpha_t}$, we deduce that $Q_s = Q_t$ is different from $\P_s(c_+)$. As $Q_s$ and $\P_s(c_+)$ are parallel by Lemma \ref{Lemma: stabilized implies parallel}, they are opposite by \cite[Lemma $18$]{DMVM11}. Similarly, $Q_t$ and $\P_t(c_+)$ are opposite. But then $Q_s = Q_t$ would be opposite to $\P_s(c_+)$ and to $\P_t(c_+)$, which is a contradiction. Thus we have $Q_s \neq Q_t$. Now if $Q_s$ and $Q_t$ are parallel, it follows from \cite[Lemma $18$]{DMVM11} that $Q_s$ and $Q_t$ are opposite.
\end{proof}

\begin{lemma}\label{Lemma: uniqueness of stabilized panels}
	Let $\Delta = (\C, \delta)$ be a Moufang building of rank $2$ and of irreducible and spherical type. Assume that all panels contain exactly three chambers. Let $\Sigma$ be an apartment and let $c_+, c_- \in \Sigma$ be the two fundamental opposite chambers. Let $(U_{\alpha})_{\alpha \in \Phi^{\Sigma}}$ be the family of corresponding root groups. Then we have
	\[ \{ c\in \C \mid \forall s\in S \, \exists s' \in S: \langle U_{\alpha_s} \cup U_{-\alpha_s} \rangle \leq \Stab( \P_{s'}(c) ) \} = \{ c_+, c_- \}. \]
\end{lemma}
\begin{proof}
	One inclusion is obvious. For the other we let $c_+ \neq d\in \C$ be a chamber such that for each $s\in S$ there exists $s' \in S$ with $\langle U_{\alpha_s} \cup U_{-\alpha_s} \rangle \leq \Stab(\P_{s'}(d))$. Note that $s' \neq t'$ for $S = \{s, t\}$ by the previous lemma.
	
	\emph{Claim $1$: $\P_r(c_+) \neq \P_{r'}(d)$ $\Rightarrow$ $\P_r(c_+)$ and $\P_{r'}(d)$ are opposite.}
	
	As both panels are stabilized by $\langle U_{\alpha_r} \cup U_{-\alpha_r} \rangle$ and no chamber in $\P_r(c_+)$ is fixed by $\langle U_{\alpha_r} \cup U_{-\alpha_r} \rangle$, Lemma \ref{Lemma: stabilized implies parallel}$(a)$ yields that the panels $\P_r(c_+)$ and $\P_{r'}(d)$ are parallel. Now \cite[Lemma $18$]{DMVM11} implies that $\P_r(c_+)$ and $\P_{r'}(d)$ are opposite.
	
	\emph{Claim $2$: $c_+$ and $d$ are opposite.}
	
	We first assume $\P_s(c_+) = \P_s(d)$ for some $s\in S$. As $U_{-\alpha_t}$ does not stabilize $\P_s(c_+)$, we infer $s=s'$, $t=t'$ and $\langle U_{\alpha_t} \cup U_{-\alpha_t} \rangle \leq \Stab(\P_t(d))$ for $S = \{s, t\}$. As $c_+ \neq d$, we infer $\P_t(c_+) \neq \P_t(d)$. Now Claim $1$ implies $\delta(c_+, d) \in \{ r_S t, r_S \}$. This is a contradiction to $\P_s(c_+) = \P_s(d)$. Thus we have $\P_s(c_+) \neq \P_s(d)$ for all $s\in S$. In particular, we have $\P_r(c_+) \neq \P_{r'}(d)$ for all $r\in S$ and Claim $1$ yields $\delta(c_+, d) \in \{r_S s', r_S\} \cap \{ r_S t', r_S \} = \{r_S\}$. Hence $c_+$ and $d$ are opposite.
	
	Now there exists $g\in U_+$ with $g.c_- = d$ (cf.\ \cite[Corollary $7.67$]{AB08}). Let $r\in S$ and suppose $r'' \in S$ with $\langle U_{\alpha_r} \cup U_{-\alpha_r} \rangle \leq \Stab(\P_{r''}(c_-))$. Then the subgroup $\langle U_{\alpha_r} \cup U_{-\alpha_r} \rangle$ stabilizes $\P_{r'}(d)$ and $\P_{r''}(c_-)$. If $m_{st} = 3$, then $r \neq r'$ and $r\neq r''$, which implies $r' = r''$. If $m_{st} \neq 3$, then $r' = r = r''$. We infer that $r' = r''$ in any case. We deduce $[g, U_{\pm \alpha_r}]. \P_{r'}(c_-) = \P_{r'}(c_-)$. Note that $[g, U_{\alpha_r}] \in N_r := \langle U_{\alpha} \mid \alpha \in \Phi_+ \backslash \{ \alpha_r \} \rangle$. As no non-trivial element of $N_r$ stabilizes $\P_r(c_-)$, we deduce $[g, U_{\alpha_r}] = 1$. We distinguish the following cases:
	\begin{enumerate}[label=$m_{st}\equal \arabic*$]
		\setcounter{enumi}{2}
		\item As $\langle U_{\alpha_r} \cup U_{-\alpha_r} \rangle$ stabilizes $\P_{r'}(d)$ and $\P_{r''}(c_-)$, they are parallel by Lemma \ref{Lemma: stabilized implies parallel}. As $r' = r''$, they cannot be opposite and \cite[Lemma $18$]{DMVM11} implies $\P_{r'}(d) = \P_{r''}(c_-)$. As this holds for all $r\in S$, we conclude $d=c_-$.
		
		\item As $[g, U_{\alpha_r}] = 1$ holds for both $r\in S$, we deduce from the commutator relations that $g\in U_{t\alpha_s} U_{s\alpha_t}$. Now $[g, U_{-\alpha_r}] \in N_r$ as before and hence $[g, U_{-\alpha_r}] = 1$. We obtain $g=1$ and hence $d = g.c_- = c_-$.
		
		\setcounter{enumi}{5}
		\item Now we use the notation from \cite{TW02}. As $[g, U_{\alpha_6}] = 1$, we deduce from the commutator relations that $g \in U_3 \cdots U_6$. As $[g, U_{\alpha_1}] = 1$, we deduce that $g\in U_1 U_2 U_4$. By \cite[$5.6$]{TW02} the expression is unique and hence $g \in U_4$. As before, we deduce $[g, U_{-6}] = 1$. But then $[u_4, u_{-6}] = [u_2, u_6] = u_4 \neq 1$. This implies $g=1$ and hence $d= = g.c_- = c_-$. \qedhere
	\end{enumerate}
\end{proof}

\subsection*{Maximal finite subgroups}

By a \emph{maximal finite subgroup} $U\leq G$ of a group $G$ we mean a finite subgroup which is not properly contained in any other finite subgroup of $G$.

\begin{theorem}\label{Theorem: Theorem 4.1 and Proposition 4.2 of CM06}
	Suppose that $(W, S)$ is $2$-complete of rank $\geq 3$. Let $\mathcal{D} = \left( G, \left( U_{\alpha} \right)_{\alpha \in \Phi} \right)$ be an RGD-system of type $(W, S)$ such that all root groups are finite and such that $H := \bigcap_{\alpha \in \Phi} N_G(U_{\alpha})$ is finite. Let $\Delta(\mathcal{D}) = (\Delta_+, \Delta_-, \delta_*)$ be the twin building associated with $\mathcal{D}$.
	\begin{enumerate}[label=(\alph*)]
		\item Let $M \leq G$ be a maximal finite subgroup. Then for $\epsilon \in \{+, -\}$ there exists a unique spherical residue $R_{\epsilon}$ of $\Delta_{\epsilon}$ which is stabilized by $M$. The residues $R_+$ and $R_-$ are maximal spherical, and opposite in $\Delta(\mathcal{D})$. In particular, we have $M = \Stab_G(R_+) \cap \Stab_G(R_-)$.
		
		\item Let $R_+ \subseteq \Delta_+$ and $R_- \subseteq \Delta_-$ be maximal spherical residues which are opposite. Then $M := \Stab_G(R_+) \cap \Stab_G(R_-)$ is a maximal finite subgroup.
	\end{enumerate}
\end{theorem}
\begin{proof}
	This follows from \cite[Corollary $3.8$, Theorem 4.1, Proposition 4.2]{CM06}.
\end{proof}

Let $\mathcal{D} = \left( G, \left( U_{\alpha} \right)_{\alpha \in \Phi} \right)$ be an RGD-system of type $(W, S)$. Let $\Sigma$ be a twin apartment of $\Delta(\mathcal{D})$ and let $R$ be a spherical residue of $\Delta(\mathcal{D})$ with $\Sigma \cap R \neq \emptyset$. Then we define
\[ L^{\Sigma}(R) := \Fix_G(\Sigma) . \langle U_{\alpha} \mid \alpha \in \Phi^{\Sigma}(R) \rangle. \]

\begin{lemma}\label{Lemma: Remy03 Section 6.2.2}
	Let $\mathcal{D} = \left( G, \left( U_{\alpha} \right)_{\alpha \in \Phi} \right)$ be an RGD-system of type $(W, S)$ and let $\Delta(\mathcal{D}) = (\Delta_+, \Delta_-, \delta_*)$ be the associated twin building. Let $R_+ \subseteq \Delta_+$ and $R_- \subseteq \Delta_-$ be opposite spherical residues and let $\Sigma$ be a twin apartment with $\Sigma \cap R_+ \neq \emptyset \neq \Sigma \cap R_-$. Then we have
	\[ L^{\Sigma}(R_+) = \Stab_G(R_+) \cap \Stab_G(R_-). \]
\end{lemma}
\begin{proof}
	One inclusion is obvious. For the other we let $g\in \Stab_G(R_+) \cap \Stab_G(R_-)$. Note that with the restriction of the codistance, $(R_+, R_-, \delta_*)$ is a Moufang twin building and $\langle U_{\alpha} \mid \alpha \in \Phi^{\Sigma}(R_+) \rangle$ acts strongly transitively on $(R_+, R_-, \delta_*)$ by \cite[Proposition $8.19$]{AB08}. But this implies that there exists $h\in \langle U_{\alpha} \mid \alpha \in \Phi^{\Sigma}(R_+) \rangle$ with $gh \in \Fix_G(\Sigma)$. This finishes the proof.
\end{proof}

\subsection*{Isomorphisms of (irreducible) RGD-systems}

\begin{definition}
	Suppose that $(W, S)$ and $(W', S')$ are irreducible Coxeter systems of finite rank. Let $\mathcal{D} = \left( G, \left( U_{\alpha} \right)_{\alpha \in \Phi} \right)$ be an RGD-system of type $(W, S)$ and let $\mathcal{D}' = \left( G', \left( U_{\alpha}' \right)_{\alpha \in \Phi'} \right)$ be an RGD-system of type $(W', S')$. Then $\mathcal{D}$ and $\mathcal{D'}$ are called \emph{isomorphic} if there exist an isomorphism $\phi: G \to G'$, an isomorphism $\pi: W \to W'$ with $\pi(S) = S'$, an element $x\in G'$ and a sign $\epsilon \in \{+, -\}$ such that
	\[ \phi\left( U_{\alpha} \right) = x U_{\epsilon \pi(\alpha)}' x^{-1} \]
	for each $\alpha \in \Phi$. If $\phi$ is as above, then we say that $\phi$ \emph{induces an isomorphism} of $\mathcal{D}$ to $\mathcal{D}'$.
\end{definition}

\begin{theorem}\label{Theorem: Theorem 2.2 in CM05b}
	Let $\mathcal{D} = \left( G, \left( U_{\alpha} \right)_{\alpha \in \Phi} \right)$ be an RGD-system of irreducible type $(W, S)$ and let $\mathcal{D}' = \left( G', \left( U_{\alpha}' \right)_{\alpha \in \Phi'} \right)$ be an RGD-system of irreducible type $(W', S')$. Suppose that $S$ and $S'$ are finite. If $\phi: G \to G'$ is an isomorphism and if there exists $x\in G'$ with
	\[ \left\{ \phi(U_{\alpha}) \mid \alpha \in \Phi(W, S) \right\} = \left\{ x U_{\alpha}' x^{-1} \mid \alpha \in \Phi(W', S') \right\}, \]
	then $\phi$ induces an isomorphism of $\mathcal{D}$ to $\mathcal{D}'$.
\end{theorem}
\begin{proof}
	This is \cite[Theorem 2.2]{CM05b}.
\end{proof}

\begin{definition}
	Let $\mathcal{D} = \left( G, \left( U_{\alpha} \right)_{\alpha \in \Phi} \right)$ be an RGD-system of type $(W, S)$. We define $G^{\mathcal{D}} := G$. We say that $\mathcal{D}$ has \emph{trivial torus} if $H = \bigcap_{\alpha \in \Phi} N_G(U_{\alpha}) = \{1\}$. Moreover, $\mathcal{D}$ is called \emph{centered} if $G = \langle U_{\alpha} \mid \alpha \in \Phi \rangle$; it is called \emph{over} $\FF_2$ if all root groups have cardinality $2$.
\end{definition}

\begin{remark}
	Note that a centered RGD-system over $\FF_2$ has trivial torus.
\end{remark}

\section{The Main result}\label{Section: Main result}

In this section we will solve the isomorphism problem for RGD-systems over $\FF_2$. We will follows the strategy used in \cite[Theorem $5.1$]{CM06}.

\begin{convention}
	In this section we let $(W, S)$ and $(W', S')$ be two $2$-complete and $\tilde{A}_2$-free Coxeter systems of finite rank at least $3$, $\mathcal{D}$ and $\mathcal{D}'$ be two centered RGD-systems over $\FF_2$ of type $(W, S)$ and $(W', S')$, respectively, and $\phi: G^{\mathcal{D}} \to G^{\mathcal{D}'}$ be an isomorphism. Then $\mathcal{D}$ and $\mathcal{D}'$ have trivial tori.
	
	We let $\Delta$ (resp.\ $\Delta'$) be the twin building associated with $\mathcal{D}$ (resp.\ $\mathcal{D}'$). We denote the root groups of $\Delta$ by $U_{\alpha}$ and the root groups of $\Delta'$ by $U_{\alpha}'$. Let $\Sigma$ be the fundamental twin apartment of $\Delta$ and let $c_{\epsilon} \in \C_{\epsilon}$ be the fundamental pair of opposite chambers of $\Delta$ for $\epsilon \in \{+, -\}$. For $s\neq t \in S$ we define $M_{st} := L^{\Sigma}(R_{\{s, t\}}(c_+)) = \langle U_{\pm \alpha_s} \cup U_{\pm \alpha_t} \rangle = \Stab_{G^{\mathcal{D}}}(R_{\{s, t\}}(c_+)) \cap \Stab_{G^{\mathcal{D}}}(R_{\{s, t\}}(c_-))$ (cf.\ Lemma \ref{Lemma: Remy03 Section 6.2.2}).
\end{convention}

\begin{lemma}\label{Lemma: phi(Mst) stabilizer of opposite residues}
	Let $s \neq t \in S$ and $\epsilon \in \{+, -\}$. Then there exists a unique spherical residue $R_{\epsilon}^{\{s, t\}}$ of $\Delta_{\epsilon}'$ which is stabilized by $\phi( M_{st} )$. Moreover, $R_+^{\{s, t\}}$ and $R_-^{\{s, t\}}$ are opposite and of rank $2$, and $\phi(M_{st}) = \Stab_{G^{\mathcal{D}'}}( R_+^{\{s, t\}} ) \cap \Stab_{G^{\mathcal{D}'}}( R_-^{\{s, t\}} )$.
\end{lemma}
\begin{proof}
	$M_{st}$ is a maximal finite subgroup by Theorem \ref{Theorem: Theorem 4.1 and Proposition 4.2 of CM06}$(b)$. Thus $\phi(M_{st})$ is a maximal finite subgroup of $G^{\mathcal{D}'}$. By Theorem \ref{Theorem: Theorem 4.1 and Proposition 4.2 of CM06}$(a)$, there exist unique spherical residues $R_+^{\{s, t\}}$ of $\Delta_+'$ and $R_-^{\{s, t\}}$ of $\Delta_-'$ which are stabilized by $\phi(M_{st})$. In particular, $\phi(M_{st}) = \Stab_{G^{\mathcal{D}'}}( R_+^{\{s, t\}} ) \cap \Stab_{G^{\mathcal{D}'}}( R_-^{\{s, t\}} )$ and the residues $R_+^{\{s, t\}}$ and $R_-^{\{s, t\}}$ are opposite in $\Delta'$ and maximal spherical, i.e.\ of rank $2$.
\end{proof}

\begin{lemma}\label{Lemma: different residues}
	For $r, s, t \in S$ pairwise distinct and $\epsilon \in \{+, -\}$ we have $R_{\epsilon}^{\{r, s\}} \neq R_{\epsilon}^{\{s, t\}}$.
\end{lemma}
\begin{proof}
	Assume by contrary $R_{\epsilon}^{\{r, s\}} = R_{\epsilon}^{\{s, t\}}$. As $R_{-\epsilon}^{\{r, s\}}$ is opposite to $R_{\epsilon}^{\{r, s\}}$, they have the same type. In particular, $R_{-\epsilon}^{\{r, s\}}$ and $R_{-\epsilon}^{\{s, t\}}$ have the same type. Now there exists an element $g\in \Stab_{G^{\mathcal{D}'}}(R_{\epsilon}^{\{s, t\}})$ with $g.R_{-\epsilon}^{\{r, s\}} = R_{-\epsilon}^{\{s, t\}}$ (cf.\ \cite[Corollary $8.32$]{AB08}). This implies $\phi(M_{st}) = g \phi(M_{rs}) g^{-1}$ and hence $M_{st} = h M_{rs} h^{-1}$ for $h := \phi^{-1}(g)$. But then we have $M_{st} \leq \Stab_{G^{\mathcal{D}}}(R_{\{s, t\}}(c_{\epsilon})) \cap \Stab_{G^{\mathcal{D}}}( h.R_{\{r, s\}}(c_{\epsilon}) )$, which is a contradiction to Theorem \ref{Theorem: Theorem 4.1 and Proposition 4.2 of CM06}$(a)$ (residues having different types are different).
\end{proof}

\begin{definition}
	Let $\Delta = (\Delta_+, \Delta_-, \delta_*)$ be a twin building of type $(W, S)$ and let $R$ and $T$ be two opposite residues of type $J$. Then $(R, T)$ is called a \emph{twin residue}. It is a basic fact that a twin residue is again a twin building of type $(\langle J \rangle, J)$.
\end{definition}

\begin{lemma}
	Let $s\neq t\in S$, $\epsilon \in \{+, -\}$ and let $\{ \overline{s}, \overline{t} \}$ be the type of $R_{\epsilon}^{\{s, t\}}$. Then there exists a twin apartment $\Sigma_{st}'$ of $\Delta'$ meeting $R_+^{\{s, t\}}$ and $R_-^{\{s, t\}}$, and $\epsilon_{st} \in \{+, -\}$ such that for $r\in \{s, t\}$ we have $\phi( U_{\pm \alpha_r} ) = U_{\pm \epsilon_{st} \alpha_{\overline{r}}}'$ and $\epsilon_{st} \alpha_{\overline{r}} \in \Phi^{\Sigma_{st}'}$.
\end{lemma}
\begin{proof}
	We define $M_{st}' := \phi(M_{st})$. We restrict the action of $M_{st}$ (resp.\ $M_{st}'$) to the twin residue $(R_{\{s, t\}}(c_+), R_{\{s, t\}}(c_-))$ (resp.\ $(R_+^{\{s, t\}}, R_-^{\{s, t\}})$). Now \cite[Proposition $8.82(b)$]{AB08} implies that $Z(M_{st})$ coincides with the kernel of the action of $M_{st}$ on the twin residue. The same is true for $M_{st}'$. As $M_{st}$ is center-free (the torus is trivial), we infer that $M_{st}$ is isomorphic to one of the following groups:
	\[ A_2(2), B_2(2), G_2(2). \]
	Note that these groups are pairwise non-isomorphic (cf.\ \cite[Theorem $37$]{St16}). Let $\Sigma_{st'}$ be a twin apartment which meets $R_+^{\{s, t\}}$ and $R_-^{\{s, t\}}$. Now \cite[Theorem $30$]{St16} implies that $\phi \vert_{M_{st}}: M_{st} \to M_{st}'$ induces an isomorphism 
	from the RGD-system $\left( M_{st}, \left( U_{\alpha} \right)_{\alpha \in \Phi^{\Sigma}(R_{\{s, t\}}(c_+))} \right)$ to the RGD-system $\left( M_{st}', \left( U_{\alpha}' \right)_{\alpha \in \Phi^{\Sigma_{st}'}(R_+^{\{s, t\}})} \right)$. This means that there exists an isomorphism $\pi: \langle s, t \rangle \to \langle \overline{s}, \overline{t} \rangle$ with $\pi(\{s, t\}) = \{ \overline{s}, \overline{t} \}$, an element $x\in M_{st}'$ and a sign $\epsilon_{st} \in \{+, -\}$ such that
	\[ \phi(U_{\alpha}) = x U_{\epsilon_{st} \pi(\alpha)}' x^{-1} \]
	where $\epsilon_{st} \pi(\alpha) \in \Phi^{\Sigma_{st}'}$. After replacing $\Sigma_{st}'$ by $x \Sigma_{st}'$, we can assume without loss of generality that $\phi\left( U_{\pm \alpha_r} \right) = U_{\pm \epsilon_{st} \alpha_{\bar{r}}}'$, where $\epsilon \in \{+, -\}$ and $r\in \{s, t\}$.
\end{proof}

\begin{lemma}\label{Lemma: triangle}
	Let $r, s, t\in S$ be pairwise distinct and $\epsilon \in \{+, -\}$. Then the set $\{ R_{\epsilon}^{\{r, s\}}, R_{\epsilon}^{\{r, t\}}, R_{\epsilon}^{\{s, t\}} \}$ is a triangle.
\end{lemma}
\begin{proof}
	We abbreviate $P_s^r := \proj_{R_{\epsilon}^{\{r, s\}}} R_{\epsilon}^{\{s, t\}}$. Note that $\{ \phi(U_{\pm \alpha_s}), \phi(U_{\pm\alpha_t}) \}$ are root groups corresponding to simple roots and their opposites in the twin apartment $\Sigma_{st}'$. We deduce that $\phi(\langle U_{\alpha_s} \cup U_{-\alpha_s} \rangle)$ stabilizes a panel $P_{\epsilon}^{\{s, t\}}$ contained in $R_{\epsilon}^{\{s, t\}}$ and, moreover, $\phi(\langle U_{\alpha_s} \cup U_{-\alpha_s} \rangle)$ acts transitively on $P_{\epsilon}^{\{s, t\}}$. By Lemma \ref{Lemma: stabilized implies parallel}$(b)$ the residue $\proj_{R_{\epsilon}^{\{r, s\}}} P_{\epsilon}^{\{s, t\}}$ is a panel, which is itself stabilized by $\phi\left( \langle U_{\alpha_s} \cup U_{-\alpha_s} \rangle \right)$. Note that $P_s^r$ contains the panel $\proj_{R_{\epsilon}^{\{r, s\}}} P_{\epsilon}^{\{s, t\}}$. By Lemma \ref{Lemma: different residues} we have $R_{\epsilon}^{\{r, s\}} \neq R_{\epsilon}^{\{s, t\}}$. Now Lemma \ref{Lemma: results about parallel residues}$(c)$ and Lemma \ref{Lemma: 2-dimensional rank 2 residues not parallel}$(a)$ yield $P_s^r = \proj_{R_{\epsilon}^{\{r, s\}}} P_{\epsilon}^{\{s, t\}}$ and (T$1$) holds. Note that $P_r^s$ is a panel which is stabilized by $\phi \langle \left( U_{\alpha_r} \cup U_{-\alpha_r} \rangle \right)$.
	
	To show that (T$2$) holds, we assume by contrary that $P_s^r$ and $P_r^s$ are parallel. Then they are opposite in $R_{\epsilon}^{\{r, s\}}$ by Lemma \ref{Lemma: no panel is stabilized by all root groups}. Note that $P_s^t$ and $P_s^r$ (resp.\ $P_r^s$ and $P_r^t$) are parallel by Lemma \ref{Lemma: results about parallel residues}$(a)$. Using Lemma \ref{Lemma: results about parallel panels}$(a)$ and Lemma \ref{Lemma: concatenation of compatible paths} we obtain that the concatenation of compatible paths from $P_s^t$ to $P_s^r$ and from $P_r^s$ to $P_r^t$ (or rather an extension by one panel in $R_{\epsilon}^{\{r, t\}}$ as in Convention \ref{Convention: triangle}) is again a compatible path. We deduce $P_r^t = \proj_{R_{\epsilon}^{\{r, t\}}} P_s^t \subseteq P_t^r$ which is a contradiction to the facts that $P_t^r$ is a panel and $P_r^t \neq P_t^r$ by Lemma \ref{Lemma: no panel is stabilized by all root groups}. Hence $P_s^r$ and $P_r^s$ are not parallel and $\{ R_{\epsilon}^{\{ r, s \}}, R_{\epsilon}^{\{ r, t \}}, R_{\epsilon}^{\{ s, t \}} \}$ is a triangle.
\end{proof}

\begin{lemma}
	For $\epsilon \in \{+, -\}$ we have $\bigcap_{x\neq y \in S} R_{\epsilon}^{\{x, y\}} = \{ d_{\epsilon} \}$ for some chamber $d_{\epsilon} \in \Delta_{\epsilon}'$. Moreover, $d_+$ and $d_-$ are opposite, $\Sigma(d_+, d_-) = \Sigma_{st}'$ and for all $s\neq t \in S$, and $\{ \phi(U_{\pm \alpha_s}), \phi(U_{\pm \alpha_t}) \}$ are the root groups corresponding to the simple roots and their opposites in $\Sigma(d_+, d_-)$.
\end{lemma}
\begin{proof}
	Let $r, s, t \in S$ be pairwise distinct. Then $\{ R_{\epsilon}^{\{ r, s \}}, R_{\epsilon}^{\{ r, t \}}, R_{\epsilon}^{\{ s, t \}} \}$ is a triangle by the previous lemma. Theorem \ref{Theorem: 2-complete and A_2 tilde free triangle contains unique chamber} implies that $R_{\epsilon}^{\{ r, s \}} \cap R_{\epsilon}^{\{ r, t \}} \cap R_{\epsilon}^{\{ s, t \}}$ is non-empty and contains a unique chamber, which we denote by $d_{\epsilon}$. Suppose that $R_{\epsilon}^{\{r, s\}} \cap R_{\epsilon}^{\{s, t\}}$ is just a chamber. Then \cite[Lemma $2.15$]{AB08} would imply that $\proj_{R_{\epsilon}^{\{r, s\}}} R_{\epsilon}^{\{s, t\}}$ is also just a chamber, which is a contradiction to the fact that $\proj_{R_{\epsilon}^{\{r, s\}}} R_{\epsilon}^{\{s, t\}}$ is a panel. This, together with Lemma \ref{Lemma: different residues}, implies that $R_{\epsilon}^{\{r, s\}} \cap R_{\epsilon}^{\{s, t\}}$ is a panel, and it coincides with $\proj_{R_{\epsilon}^{\{s, t\}}} R_{\epsilon}^{\{r, s\}}$. Let $s' \in S'$ be such that $R_{\epsilon}^{\{r, s\}} \cap R_{\epsilon}^{\{s, t\}} = \P_{s'}(d_{\epsilon})$ (similarly for $r'$ and $t'$). Then $r', s', t' \in S'$ are pairwise distinct. Note that $\delta_*'(d_+, d_-) \in \langle s', t' \rangle$, as $R_+^{\{s, t\}}$ and $R_-^{\{s, t\}}$ are opposite and of type $\{s', t'\}$. Using \cite[Exercise $2.26$]{AB08} we deduce
	\[ \delta_*'(d_+, d_-) \in \langle r', s' \rangle \cap \langle r', t' \rangle \cap \langle s', t' \rangle = \{1\}. \]
	This implies that $d_+$ and $d_-$ are opposite. As $\P_{s'}(d_{\epsilon}) = \proj_{R_{\epsilon}^{\{s, t\}}} R_{\epsilon}^{\{r, s\}}$ is stabilized by $\phi(\langle U_{\alpha_s} \cup U_{-\alpha_s} \rangle)$ (cf.\ proof of Lemma \ref{Lemma: triangle}) and $\P_{t'}(d_{\epsilon}) = \proj_{R_{\epsilon}^{\{s, t\}}} R_{\epsilon}^{\{r, t\}}$ is stabilized by $\phi(\langle U_{\alpha_t} \cup U_{-\alpha_t} \rangle)$, it follows from Lemma \ref{Lemma: uniqueness of stabilized panels} that $d_+, d_- \in \Sigma_{st}'$ for all $s\neq t \in S$ and, hence, $\Sigma_{st}' = \Sigma(d_+, d_-)$. Moreover, $\{ \phi(U_{\pm \alpha_s}), \phi(U_{\pm \alpha_t}) \}$ are the root groups corresponding to the simple roots and their opposites in $\Sigma(d_+, d_-)$, i.e.\ there exists $\epsilon_{st} \in \{+, -\}$ and a bijection $\pi: \langle s, t \rangle \to \langle s', t' \rangle$ with $\pi(\{s, t\}) = \{s', t'\}$ such that $\phi(U_{\pm \alpha_r}) = U_{\pm \epsilon_{st} \alpha_{\pi(r)}}'$ for $r\in \{s, t\}$ where $\pm \epsilon_{st} \alpha_{\pi(r)} \in \Phi^{\Sigma(d_+, d_-)}$.
	
	\emph{Claim: If $a, b, c, d \in S$ are pairwise distinct with $R_{\epsilon}^{\{a, b\}} \cap R_{\epsilon}^{\{a, c\}} \cap R_{\epsilon}^{\{b, c\}} = \{ d_{\epsilon} \}$, then $d_{\epsilon} \in R_{\epsilon}^{\{c, d\}}$.}
	
	As before, $R_{\epsilon}^{\{b, c\}} \cap R_{\epsilon}^{\{b, d\}} \cap R_{\epsilon}^{\{c, d\}}$ is non-empty and contains a unique chamber, which we denote by $d_{\epsilon}'$. Note that $R_{\epsilon}^{\{a, b\}} \cap R_{\epsilon}^{\{a, d\}} \cap R_{\epsilon}^{\{b, d\}} \neq \emptyset$. Assume $d_{\epsilon}' \neq d_{\epsilon}$. Then, by Lemma \ref{Lemma: uniqueness of stabilized panels}, we know that $d_{\epsilon}'$ and $d_{\epsilon}$ are opposite in $R_{\epsilon}^{\{b, c\}}$. If $\ell(c_{bd}, c_{ab}) \geq 2$ for all $c_{bd} \in R_{\epsilon}^{\{b, d\}} \cap R_{\epsilon}^{\{b, c\}}$ and $c_{ab} \in R_{\epsilon}^{\{a, b\}} \cap R_{\epsilon}^{\{b, c\}}$, then it follows from Lemma \ref{Lemma: BiCoxGrowth2.8} that $R_{\epsilon}^{\{a, b\}} \cap R_{\epsilon}^{\{b, d\}} = \emptyset$ which is a contradiction. Thus we can assume that there exist $c_{bd} \in R_{\epsilon}^{\{b, d\}} \cap R_{\epsilon}^{\{b, c\}}$ and $c_{ab} \in R_{\epsilon}^{\{a, b\}} \cap R_{\epsilon}^{\{b, c\}}$ with $\ell(c_{bd}, c_{ab}) = 1$. We deduce $\delta_{\epsilon}'(c_{bd}, c_{ab}) = c'$. Let $e\in R_{\epsilon}^{\{a, b\}} \cap R_{\epsilon}^{\{a, d\}} \cap R_{\epsilon}^{\{b, d\}}$. Then we have $\delta_{\epsilon}'(e, c_{bd}) \in \langle b', d' \rangle$ and $\delta_{\epsilon}'(e, c_{ab}) \in \langle a', b' \rangle$. In particular, we have $c' = \delta_{\epsilon}'(c_{bd}, c_{ab}) \in \langle a', b', d' \rangle$ and \cite[Exercise $2.26$]{AB08} implies $c' \in \{ a', b', d' \}$. As $a', b', c'$ as well as $b', c', d'$ are paiswise distinct, we obtain a contradiction.
	
	Now we are in the position to prove the claim. Let $u \neq v \in S$. Note that it suffices to show $d_{\epsilon} \in R_{\epsilon}^{\{u, v\}}$. For $u, v \in \{r, s, t\}$ there is nothing to show. Thus we can assume without loss of generality that $v\notin \{r, s, t\}$. If $u \in \{r, s, t\}$, then $d_{\epsilon} \in R_{\epsilon}^{\{u, v\}}$ follows from the claim and we can assume $u, v \notin \{r, s, t\}$. The claim applied to $r, s, t, u$ implies $R_{\epsilon}^{\{u, s\}} \cap R_{\epsilon}^{\{u, t\}} \cap R_{\epsilon}^{\{s, t\}} = \{ d_{\epsilon} \}$. Applying the claim now to $s, t, u, v$, we infer $d_{\epsilon} \in R_{\epsilon}^{\{u, v\}}$.
\end{proof}

\begin{theorem}\label{Theorem: Main theorem}
	$\phi$ induces an isomorphism between the RGD-systems $\mathcal{D}$ and $\mathcal{D}'$.
\end{theorem}
\begin{proof}
	Recall from the proof of the previous lemma that for all $s\neq t \in S$ there exists an isomorphism $\pi_{st}: \langle s, t \rangle \to \langle s', t' \rangle$ with $\pi(\{s, t\}) = \{s', t'\}$ and $\epsilon_{st} \in \{+, -\}$ such that $\phi(U_{\pm \alpha_s}) = U_{\pm \epsilon_{st} \alpha_{\pi_{st}(s)}}'$ and $\phi(U_{\pm \alpha_t}) = U_{\pm \epsilon_{st} \alpha_{\pi_{st}(t)}}'$ where $\pm \epsilon_{st} \alpha_{\pi_{st}(s)}, \pm \epsilon_{st} \alpha_{\pi_{st}(t)} \in \Phi^{\Sigma(d_+, d_-)}$.
	
	Let $a, b, c, d \in S$.	We first see that $\epsilon_{ab} = \epsilon_{ac}$ for all $a, b, c \in S$ pairwise distinct. Now it follows directly, that $\epsilon_{ab} = \epsilon_{ac} = \epsilon_{cd}$. We define $\epsilon := \epsilon_{st}$ for some $s\neq t \in S$. Note that $\pi_{st}$ extends to an isomorphism $\pi: W \to W'$ with $\pi(S) = S'$.
	
	\emph{Claim: We have $\{ \phi(U_{\alpha}) \mid \alpha \in \Phi^{\Sigma} \} = \{ U_{\alpha}' \mid \alpha \in \Phi^{\Sigma(d_+, d_-)} \}$.}
	
	Note that for each $u\in U_{\pm\alpha_s}$ with $u\neq 1$ there are unique elements $u', u'' \in U_{\mp\alpha_s}$ such that $m(u) = u'uu''$ conjugates $U_{\beta}$ onto $U_{s\beta}$ for each $\beta \in \Phi$. This implies
	\[ U_{-\epsilon \alpha_{\pi(s)}}' = \phi( U_{-\alpha_s} ) = \phi( m(u)^{-1} U_{\alpha_s} m(u) ) = \phi(m(u))^{-1} U_{\epsilon \alpha_{\pi(s)}}' \phi(m(u)). \] 
	We infer $\phi(m(u)) = m(\phi(u))$. Let $\alpha \in \Phi$ with $\alpha = s_1 \cdots s_k \alpha_s$ for some $s_1, \ldots, s_k, s\in S$, then $\phi$ maps $U_{\alpha} = U_{\alpha_s}^{m(u_k) \cdots m(u_1)}$ to $(U_{\epsilon \alpha_{\pi(s)}}')^{m(\phi(u_k)) \cdots m(\phi(u_1))} = U_{\epsilon\pi(s_1) \cdots \pi(s_k) \alpha_{\pi(s)}}$.
	
	Let $\Sigma'$ be the fundamental twin apartment of $\Delta'$. As $G^{\mathcal{D}'}$ acts strongly transitively on $\Delta'$, there exists $x\in G^{\mathcal{D}'}$ with $\{ U_{\alpha}' \mid \alpha \in \Phi^{\Sigma(d_+, d_-)} \} = \{ xU_{\alpha}' x^{-1} \mid \alpha \in \Phi^{\Sigma'} \}$. Now the claim follows from the claim together with Theorem \ref{Theorem: Theorem 2.2 in CM05b}.
\end{proof}

\bibliographystyle{amsalpha}
\bibliography{references}

\providecommand{\bysame}{\leavevmode\hbox to3em{\hrulefill}\thinspace}
\providecommand{\MR}{\relax\ifhmode\unskip\space\fi MR }
\providecommand{\MRhref}[2]{%
  \href{http://www.ams.org/mathscinet-getitem?mr=#1}{#2}
}
\providecommand{\href}[2]{#2}
\begin{thebibliography}{DMVM12}

\bibitem[AB08]{AB08}
Peter Abramenko and Kenneth~S. Brown, \emph{Buildings}, Graduate Texts in
  Mathematics, vol. 248, Springer, New York, 2008, Theory and applications.
  \MR{2439729}

\bibitem[Bis22]{Bi22}
Sebastian Bischof, \emph{On commutator relations in 2-spherical {RGD}-systems},
  Comm. Algebra \textbf{50} (2022), no.~2, 751--769. \MR{4375537}

\bibitem[Bis23]{BiDiss}
\bysame, \emph{Construction of {RGD}-systems of type $(4,4,4)$ over
  $\mathbb{F}_2$}, PhD thesis, Justus-Liebig-Universität Giessen, 2023.

\bibitem[Bis24a]{BiConstruction}
Sebastian Bischof, \emph{Construction of {C}ommutator {B}lueprints},
  https://arxiv.org/abs/2407.15506, 2024.

\bibitem[Bis24b]{BiCoxGrowth}
\bysame, \emph{On {G}rowth {F}unctions of {C}oxeter {G}roups},
  https://arxiv.org/abs/2405.10617, 2024.

\bibitem[BM23]{BM23}
Sebastian Bischof and Bernhard Mühlherr, \emph{Isometries of wall-connected
  twin buildings}, Advances in Geometry \textbf{23} (2023), no.~3, 371--388.

\bibitem[Bou02]{Bo68}
Nicolas Bourbaki, \emph{Lie groups and {L}ie algebras. {C}hapters 4--6},
  Elements of Mathematics (Berlin), Springer-Verlag, Berlin, 2002, Translated
  from the 1968 French original by Andrew Pressley. \MR{1890629}

\bibitem[Cap07]{Ca06}
Pierre-Emmanuel Caprace, \emph{On 2-spherical {K}ac-{M}oody groups and their
  central extensions}, Forum Math. \textbf{19} (2007), no.~5, 763--781.
  \MR{2350773}

\bibitem[Cap09]{Caprace_Diss}
\bysame, \emph{``{A}bstract'' homomorphisms of split {K}ac-{M}oody groups},
  Mem. Amer. Math. Soc. \textbf{198} (2009), no.~924, xvi+84. \MR{2499773}

\bibitem[CM05]{CM05b}
Pierre-Emmanuel Caprace and Bernhard M\"{u}hlherr, \emph{Isomorphisms of
  {K}ac-{M}oody groups}, Invent. Math. \textbf{161} (2005), no.~2, 361--388.
  \MR{2180452}

\bibitem[CM06]{CM06}
\bysame, \emph{Isomorphisms of {K}ac-{M}oody groups which preserve bounded
  subgroups}, Adv. Math. \textbf{206} (2006), no.~1, 250--278. \MR{2261755}

\bibitem[DMVM12]{DMVM11}
Alice Devillers, Bernhard M\"{u}hlherr, and Hendrik Van~Maldeghem,
  \emph{Codistances of 3-spherical buildings}, Math. Ann. \textbf{354} (2012),
  no.~1, 297--329. \MR{2957628}

\bibitem[Fel98]{Fe98}
A.~A. Felikson, \emph{Coxeter decompositions of hyperbolic polygons}, European
  J. Combin. \textbf{19} (1998), no.~7, 801--817. \MR{1649962}

\bibitem[KP87]{KP87}
Victor~G. Kac and Dale~H. Peterson, \emph{On geometric invariant theory for
  infinite-dimensional groups}, Algebraic groups {U}trecht 1986, Lecture Notes
  in Math., vol. 1271, Springer, Berlin, 1987, pp.~109--142. \MR{911137}

\bibitem[MPW15]{MPW15}
Bernhard M\"{u}hlherr, Holger~P. Petersson, and Richard~M. Weiss, \emph{Descent
  in buildings}, Annals of Mathematics Studies, vol. 190, Princeton University
  Press, Princeton, NJ, 2015. \MR{3364836}

\bibitem[Ste16]{St16}
Robert Steinberg, \emph{Lectures on {C}hevalley groups}, corrected ed.,
  University Lecture Series, vol.~66, American Mathematical Society,
  Providence, RI, 2016, Notes prepared by John Faulkner and Robert Wilson, With
  a foreword by Robert R. Snapp. \MR{3616493}

\bibitem[TW02]{TW02}
Jacques Tits and Richard~M. Weiss, \emph{Moufang polygons}, Springer Monographs
  in Mathematics, Springer-Verlag, Berlin, 2002. \MR{1938841}

\end{thebibliography}
\end{document}